\documentclass[12pt,a4paper]{article}
\usepackage{lscape}
\usepackage{xcolor}
\usepackage[latin1]{inputenc}
\usepackage[english]{babel}
\usepackage{amsmath}
\usepackage{amsfonts}
\usepackage{amssymb}
\usepackage{latexsym}

\usepackage{verbatim}

\usepackage{parskip}
\usepackage{fullpage}
\usepackage[pdftex]{graphicx}
\usepackage{enumerate}
\numberwithin{equation}{section}

\usepackage{amsthm}

\DeclareMathOperator{\e}{e}

\newtheorem{definicao}{Definition}[section]
\newtheorem{teorema}{Theorem}[section]

\newtheorem{corolario}[teorema]{Corollary}
\newtheorem{lema}[teorema]{Lemma}
\newtheorem{rem}[teorema]{Remark}

\newcounter{exemplo}[section]
\newcommand{\exemplo}{\stepcounter{exemplo}
	\noindent\textbf{Example \arabic{section}.\arabic{exemplo}. }}

\newcommand{\er}{\mathbb{R}}
\newcommand{\en}{\mathbb{N}}

\newcommand{\Ne}{\mathrm{e}}
\newcommand{\Es}{\hspace{0.5cm}}
\newcommand{\ov}{\overline}
\newcommand{\un}{\underline}
\newcommand{\dst}{\displaystyle}
\newcommand{\til}[1]{\widetilde{#1}}
\begin{document}
	
	\begin{center}
		{\bf {\Large 		Convergence of asymptotic systems in Cohen--Grossberg neural network models with unbounded delays
}}
	\end{center}

		\begin{center}
		A. Elmwafy$^\ddag$, Jos\'{e} J. Oliveira$^*$, C\'esar M. Silva$^\dag$
	\end{center}

\begin{center}
	($\ddag$) Centro de Matem\'{a}tica e Aplica\c{c}\~{o}es (CMA-UBI),\\
	Universidade da Beira Interior, 6201-001 Covilh\~{a}, Portugal\\
	e-mail: ahmed.elmwafy@ubi.pt\\
	
	($*$) Centro de Matem\'{a}tica (CMAT), Departamento de Matem\'{a}tica,\\
	Universidade do Minho, Campus de Gualtar, 4710-057 Braga, Portugal\\
	e-mail: jjoliveira@math.uminho.pt\\

	($\dag$) Centro de Matem\'{a}tica e Aplica\c{c}\~{o}es (CMA-UBI), Departamento de Matem\'{a}tica,\\
	Universidade da Beira Interior, 6201-001 Covilh\~{a}, Portugal\\
	e-mail: csilva@ubi.pt\\
	
\end{center}

	\begin{abstract}
			In this paper, we investigate the convergence of asymptotic systems in non-autonomous Cohen--Grossberg neural network models, which include both infinite discrete time-varying and distributed delays. We derive stability results under conditions where the non-delay terms asymptotically dominate the delay terms. Several examples and a numerical simulation are provided to illustrate the significance and novelty of the main result.
	\end{abstract}
	
	
	\noindent
	{\it Mathematics Subject Classification System 2020}: 34K14, 34K20, 34K25, 34K60, 92B20.

	\section{Introduction}\label{introducao}
	In the past few decades, neural networks have garnered significant attention due to their versatile applications such as in image and signal processing \cite{aizenberg+aizenberg+hiltner+moraga+meyer}, pattern recognition \cite{yao1990model}, optimization \cite{park2005lmi}, and content-addressable memory \cite{hartman}.

 Among the various neural network models that have been extensively investigated and applied, Cohen--Grossberg type models play a pivotal role. In their pioneering work \cite{cohen1983absolute}, Cohen and Grossberg introduced and investigated the stability of the following system of ordinary differential equations,
\begin{align}
x_i'(t)= -a_i(x_i(t))\left[b_i(x_i(t))-\sum_{j=1}^{n} c_{ij}f_j(x_j(t))+I_i\right],\hspace{0.2cm} t\geq 0,\hspace{0.1cm} i=1,\dots,n
\end{align}
where $n$ is a natural number indicating the number of neurons and, for each $i,j=1,\dots,n$, $x_i(t)$ represents the $i$th neuron state at time t, $a_i(u)$ denotes the amplification function, $b_i(u)$ is the self-signal function, $f_j(u)$ is the activation function, $c_{ij}$ represents the strength of connectivity between neurons $i$ and $j$, $I_i$ denotes the external input to the system.

Since their introduction, Cohen--Grossberg neural network (CGNN) models have become a prominent subject of investigation. The dynamical properties of CGNNs, such as stability, instability, and periodic oscillation, have been extensively studied for both theoretical and practical applications. For instance, in \cite{eliasmith2005unified}, the author examined the global convergence of the model to ensure that the trajectories of the network  do not exhibit chaotic behavior, thereby maintaining its functionality as an associative memory or optimization solver. Furthermore, the globally convergent dynamics indicate that the neural network algorithm will ensure convergence to an optimal solution from any initial guess when used as an optimization solver.

Several significant studies have been conducted in this area. For an autonomous CGNN model,  sufficient conditions were given for the coexistence and local $\mu-$stability of multiple equilibrium points \cite{lui+tan+xu}. In \cite{aouiti2020new}, the authors established sufficient conditions for the existence and global exponential stability of an almost automorphic solution of an interval general Cohen--Grossberg Bidirectional Associative Memory neural network with finite discrete delays and infinite distributed delays. In \cite{chen2019multiple}, the authors studied the $\mu-$stability and instability of an equilibrium point of a CGNN model with infinite discrete time-varying delays. 
 The existence and exponential stability of a periodic solution of a  general high-order CGNN model with infinite discrete time-varying and distributed delays were studied in \cite{submitted}. The existence of multi-periodic solutions of a generalized CGNN was studied in \cite{wang+huang}. 
 The existence and exponential stability of an almost periodic solution of a CGNN with infinite distributed delays were explored in \cite{xiang+cao}. In \cite{chen2010global}, the global exponential stability of a periodic solution of a delayed CGNN  was investigated but in the case of discontinuous activation functions. Additionally, a stochastic CGNN with delays was studied in \cite{wu+yang+ren}, and criteria for the exponential stability of a discrete time high-order CGNN model with impulses were established in \cite{dong2021global}.

To our knowledge, there are few studies on the global convergence of systems  in the context of neural network models \cite{oliveira2017convergence,xiao+zhang,yuan+huang+hu+liu,yuan+yuan+huang+hu,zhou+li+zhang}. As illustrated by a simple numerical example in \cite{yuan+yuan+huang+hu}, the dynamic behavior of a system is generally not determined by the dynamics of its asymptotic systems (see Definition \ref{def:asymptotic-system}). Therefore, it is valuable to identify situations where the dynamics of a system can be inferred by studying the dynamics of one of its asymptotic systems. In \cite{oliveira2017convergence,xiao+zhang,yuan+yuan+huang+hu,zhou+li+zhang}, sufficient conditions are provided to ensure the global convergence of systems in various delayed Hopfield neural network models. In \cite{yuan+huang+hu+liu}, the convergence of systems was examined for the following low-order CGNN with finite discrete time-varying delays:
 
\begin{eqnarray*} \label{model-yuan+huang+hu+liu}
 	x'_i(t)&=&a_i(x_i(t))\bigg[-b_i(t,x_i(t))+\sum_{j=1}^{n}c_{ij}(t)f_j\big(x_j(t)\big)+\sum_{j=1}^nb_{ij}(t)f_j(t-\tau_{ij}(t))+I_i(t)\bigg],
 \end{eqnarray*}
where the coefficients asymptotically  converge to periodic functions. Motivated by these studies, in this paper we investigate the convergence of systems in the following generalized high-order CGNN model with both infinite discrete time-varying and distributed delays:
\begin{eqnarray} \label{1}
x'_i(t)&=&a_i(t,x_i(t))\bigg[-b_i(t,x_i(t))+F_i\left(\dst\sum_{p=1}^{P}\sum_{j,l=1}^{n}c_{ijlp}(t)h_{ijlp}\big(x_j(t-\tau_{ijp}(t)),x_l(t-\til{\tau}_{ilp}(t))\big)\right.,\nonumber\\
& &\dst\sum_{p=1}^{P}\sum_{j,l=1}^{n}d_{ijlp}(t)f_{ijlp}\left(\int_{-\infty}^0g_{ijp}(x_j(t+s))d\eta_{ijp}(s),\int_{-\infty}^0\til{g}_{ilp}(x_l(t+s))d{\til{\eta}}_{ilp}(s)\right)\Bigg)\nonumber\\
& &+I_i(t)\bigg],\Es t\geq0,\,i=1,\ldots,n,
\end{eqnarray} 
where $n,P\in\en$, and $a_i:[0,+\infty)\times\er\to(0,+\infty)$, $b_i:[0,+\infty)\times\er\to\er$,  $c_{ijlp},d_{ijlp},I_i:[0,+\infty)\to\er$, $\tau_{ijp}, \til \tau_{ilp}:[0,+\infty)\to[0,+\infty)$,  $F_i, h_{ijlp},f_{ijlp}:\er^2\to\er$, $g_{ijp},\til g_{ilp}:\er\to\er$ are continuous functions, and $\eta_{ijp},\til \eta_{ilp}:(-\infty,0]\to\er$ are non-decreasing bounded functions such that $\eta_{ijp}(0)-\eta_{ijp}(-\infty)=\til\eta_{ilp}(0)-\til\eta_{ilp}(-\infty)=1$, for each $i,j,l=1,\ldots,n$, $p=1,\ldots,P$.

We note that model \eqref{1} is sufficiently general to encompass both low-order \cite{yuan+huang+hu+liu} and high-order \cite{submitted} CGNN models, Hopfield models \cite{oliveira2017convergence}, and static models \cite{ncube2020existence}.

The main objective is to achieve the global stability of model \eqref{1} through the stability of one of its asymptotic systems, which may be autonomous, periodic, or almost periodic, and thus easier to study than the original system \eqref{1}. For instance, model \eqref{1} itself may not be periodic, but it can have a periodic asymptotic system. The goal is to provide sufficient conditions to ensure that all solutions of \eqref{1} converge to a periodic function if the asymptotic system has a globally attractive periodic solution. The numerical example presented at the end of the paper illustrates this scenario.

Now, we provide an outline of the contents of this paper. Following the Introduction, Section 2 is a preliminary section where we introduce our notation and hypotheses. In Section 3, we investigate the global convergence of systems in CGNN type models. Section 4 includes some applications of the main result to some low-order and high-order CGNN models and offers a meaningful comparison with previous studies in the literature. In Section 5, we present a numerical simulation to illustrate the effectiveness and easier application of the main theoretical result. Finally, in Section 6, we conclude with a brief summary of the main novelties presented in this paper.

	\section{Preliminaries and model description}\label{not+basicResults}

% aut9hor/title details
%\author{Ahmed Elmwafy (\href{mailto:ahmedelmwafy@aims.ac.za}{ahmedelmwafy@aims.ac.za})}
% \title{Course Title: Assignment X}

% a normal section looks like this
In the present paper, we deal with the n-dimensional vector space $\er^n$, we consider the space of all bounded continuous functions $\phi:(-\infty,0]\rightarrow \er^n$ to be denoted by $BC=BC((-\infty,0];\er^n)$, equipped by the norm $\|\phi\|=\sup\limits_{s\leq 0} |\phi(s)|$, such that $|.|$ is the standard maximum norm in $\er^n$ $|x|=\max \{|x_i| \,: i=1,\dots, n\}$ for $x=(x_1,\dots,x_n)\in\mathbb{R}^n$. For a real sequence $(u_n)_{n\in \en}$, we write $u_n\nearrow +\infty$, to indicate that $u_n$ is an increasing sequence such that $\lim\limits_{n\to +\infty}u_n=+\infty$. A vector $x=(x_1,\dots,x_n)\in \er^n$ is said to be positive, denoted by $x>0$, if $x_i>0$ for all $i=1,\dots,n$.

For an open set $D$ in $BC$ and a continuous function $f:[0,\infty)\times D\rightarrow \er^n$, we consider the general setting of the retarded functional differential equation given by 
 \begin{align}\label{FDE}
\begin{split}
    x'(t)=f(t,x_t), \Es t\geq 0,
\end{split}
\end{align}
 where $x_t$ is the function $x_t:(-\infty,0]\rightarrow \er^n$ given by $x_t(s)=x(t+s)$ for $s\leq 0$ and $t\geq 0$. A solution of \eqref{FDE} on an interval $I$ in $\er$ is defined to be a function $x:(-\infty,\sup I)\rightarrow \er^n$ such that $x_t$ is in $D$, $x(t)$ is continuous differentiable, and \eqref{FDE} holds on $I$(see \cite{hale1978phase}). 

 It is generally known that the Banach space $BC$ is not a convenient phase space for \eqref{FDE} according to \cite{haddock1988precompactness,hale1978phase}, and so the typical results about existence, uniqueness, and continuous dependency of solutions are not accessible. Thus, we consider the following Banach space 
\begin{eqnarray*}
    {UC}_{g}=\left\{\phi \in C\left((-\infty,0];\er^n\right):\, \sup\limits_{s\leq 0}\frac{|\phi(s)|}{g(s)}< +\infty, \frac{\phi(s)}{g(s)} \text{ is uniformly continuous on } (-\infty,0]\right\},
\end{eqnarray*}
 equipped with the norm $\|\phi\|_g=\sup\limits_{s\leq 0}\frac{|\phi(s)|}{g(s)},$
 where $g:(-\infty,0]\rightarrow [1,+\infty)$ is a function satisfying the following conditions:
 \begin{enumerate}[(g1).]
     \item $g$ is a non-increasing continuous function with $g(0)=1$;
     \item $\lim\limits_{u\rightarrow 0^-}\frac{g(u+s)}{g(s)}=1$ uniformly on $(-\infty,0]$;
     \item $g(s)\rightarrow +\infty $ as $s\rightarrow -\infty.$
 \end{enumerate}
 For more details, see \cite{hale1978phase}.

As $BC\subseteq{UC}_{g}$, then we consider the space $BC$ with the norm $\|.\|_g$. Considering the functional differential equation \eqref{FDE} in the phase space ${UC}_{g}$, for a function $g$ that satisfies (g1)-(g3) as mentioned above, the continuity of the function $f$ assures the existence of solutions of \eqref{FDE} with initial condition $$x_{t_0}=\phi,$$ for $t_0\geq 0$ and $\phi\in {UC}_{g}$ \cite{hale1978phase}.

%we consider the space $BC=BC((-\infty,0];\mathbb{R}^n)$ of bounded and continuous functions $\phi:(-\infty,0]\to \mathbb{R}^n$, equipped with the norm $||\phi||=\dst\sup_{s\leq 0}|\phi(s)|$, where $|.|$ is the maximum norm $\in\mathbb{R}^n$ which means that $\max \{|x_i| \,: i=1,\dots, n\}$ for $x=(x_1,\dots,x_n)\in\mathbb{R}^n$. For a real sequence $(u_n)_{n\in \en}$, we write $u_n\nearrow +\infty$, we indicate that $u_n$ is an increasing sequence such that $\lim\limits_{n\to +\infty}u_n=+\infty$. A vector $x=(x_1,\dots,x_n)\in \er^n$ is said to be positive denoted by $x>0$ if $x_i>0$ for all $i=1,\dots,n$.

%For an open set $D\subseteq BC$ and a continuous function $f:[0,+\infty)\times D\to \er^n$, consider the functional differential equation in general setting given by \begin{eqnarray}\label{FDE}  x'(t)=f(t,x)_t),\Es t\geq 0\end{eqnarray}where, $x_t$ is the function $x_t:(-\infty,0]\to \er^n$ defined by $x_t(s)=x(t+s)$ for $s\leq 0$.

%To guarantee that the standard existence, uniqueness, and continuous dependence results are available, we consider the following Banach space:

%\newpage

Here, we restrict our study to solutions of \eqref{1} with bounded initial condition, i.e.
\begin{align} \label{1-IC}
	\begin{split}
		x_{t_0}=\phi,
	\end{split} 
\end{align}  
for some $t_0\geq0$ and $\phi\in BC$.

For the system \eqref{1}, the following hypotheses will be considered:
\begin{enumerate}[H1.]
	\item For each $i,j,l=1,\dots,n$, $p=1,\ldots,P$, and $u\in\er$, the functions $c_{ijlp},d_{ijlp},I_i:[0,+\infty)\to\er$, $b_i(\cdot,u):[0,+\infty)\to\er$ are bounded;
 
	\item For each $i=1,\ldots,n$, there exist $\un{a}_i,\ov{a}_i>0$ and a function $A_i:[0,+\infty)\to\er$ such that
	$$
	   \un{a}_i\leq a_i(t,v)\leq\ov{a}_i\Es\text{and}\Es A_i(t)a_i^2(t,v)\leq\frac{\partial}{\partial t}a_i(t,v),\Es t\geq 0,\,v\in\er;
	$$
 
	\item For each $i=1,\dots,n$, there exists a continuous function $\beta_i:[0,+\infty)\rightarrow[0,+\infty)$  such that
	$$
	  \frac{b_i(t,u)-b_i(t,v)}{u-v}\geq \beta_i(t) \Es t\geq 0,\,  u,v \in\mathbb{R}\text{ with } u\neq v;
	$$
	\item For each $i,j,l=1,\dots,n$ and $p=1,\dots,P$,
	$$
	  \lim_{t\to+\infty}\left(t-\tau_{ijp}(t)\right)= \lim_{t\to+\infty}\left(t-\til\tau_{ilp}(t)\right)=+\infty;
	$$
	\item For each $i,j,l=1,\dots,n$ and $p=1,\dots,P$, there are  $\gamma^{(1)}_{ijlp},\gamma^{(2)}_{ijlp},\mu^{(1)}_{ijlp},\mu^{(2)}_{ijlp}>0$  such that 
	\begin{eqnarray}
	  |h_{ijlp}(u_1,u_2)-h_{ijlp}(v_1,v_2)|\leq \gamma^{(1)}_{ijlp}|u_1-v_1|+\gamma^{(2)}_{ijlp}|u_2-v_2| \nonumber\\
	   |f_{ijlp}(u_1,u_2)-f_{ijlp}(v_1,v_2)|\leq \mu^{(1)}_{ijlp}|u_1-v_1|+\mu^{(2)}_{ijlp}|u_2-v_2|\nonumber
	\end{eqnarray}
	for all $u_1,u_2,v_1,v_2\in\er$;
	\item For each $i,j=1,\dots,n$ and $p=1,\dots,P$, there are  $\xi_{ijp},\til\xi_{ijp},\zeta_{i}, \varsigma_i>0$ such that 
	$$
	\begin{array}{ll}
		|g_{ijp}(u_1)-g_{ijp}(v_1)|\leq \xi_{ijp}|u_1-v_1|,
		\Es|\til{g}_{ijp}(u_1)-\til{g}_{ijp}(v_1)|\leq \til{\xi}_{ijp}|u_1-v_1|,\\
		|F_i(u_1,u_2)-F_i(v_1,v_2)|\leq \zeta_i|u_1-v_1|+\varsigma_{i}|u_2-v_2|,
	\end{array}
	$$
	for all $u_1,u_2,v_1,v_2\in\er$.
	\item There exists $d=(d_1,\dots,d_n)>0$ such that, for each $i=1,\dots,n$,
	\begin{eqnarray}
	\limsup_{t\rightarrow +\infty} \bigg[-\un{a}_i\big(\beta_i(t)+A_i(t)\big)+\lefteqn{\sum_{p=1}^{P}\sum_{j,l=1}^{n}\Bigg(\zeta_i|c_{ijlp}(t)|\left(\ov{a}_j\frac{d_j}{d_i}\gamma_{ijlp}^{(1)}+\ov{a}_l\frac{d_l}{d_i}\gamma_{ijlp}^{(2)}\right)}\nonumber
		\\
		+&\dst\varsigma_i|d_{ijlp}(t)|\left(\ov{a}_j\dst\frac{d_j}{d_i}\xi_{ijp}\mu_{ijlp}^{(1)}+\ov{a}_l\frac{d_l}{d_i}\til{\xi}_{ilp}\mu_{ijlp}^{(2)}\right)\Bigg)\bigg]<0.\nonumber
	\end{eqnarray}
\end{enumerate}

\begin{rem}
  Note that, if $a_i(t,u)\equiv a_i(u)$ for all $i=1,\ldots,n$, $u\in\er$, and $t\geq0$, then $A_i(t)\equiv0$ fits the second condition of hypothesis H2..
\end{rem}

Denoting, for all $i=1,\ldots,n$, $t\geq0$, and $\varphi=(\varphi_1,\ldots,\varphi_n)\in BC$,
\begin{eqnarray} \label{def-U}
	\begin{array}{rcl}
\mathcal{U}_i(t,\varphi)&=&\dst\sum_{p=1}^{P}\sum_{j,l=1}^{n}c_{ijlp}(t)h_{ijlp}\big(\varphi_j(-\tau_{ijp}(t)),\varphi_l(-\til{\tau}_{ilp}(t))\big),\\
\mathcal{V}_i(t,\varphi)&=&\dst\sum_{p=1}^{P}\sum_{j,l=1}^{n} d_{ijlp}(t)f_{ijlp}\left(\int_{-\infty}^0g_{ijp}(\varphi_j(s))d\eta_{ijp}(s),\int_{-\infty}^0\til{g}_{ilp}(\varphi_l(s))d{\til \eta}_{ilp}(s)\right),\Es
\end{array}
\end{eqnarray} 
system \eqref{1} can be written as
$$
  x'_i(t)=a_i(t,x_i(t))\big[-b_i(t,x_i(t))+F_i\left(\mathcal{U}_i(t,x_t),\mathcal{V}_i(t,x_t)\right)+I_i(t)\big],\Es t\geq0,\,i=1,\ldots,n.
$$
%where $x_t$ is defined by $x_t(s)=x(t+s)$ for $s\leq0$ and $t\geq 0$.

\begin{definicao}\label{def:asymptotic-system}
The system 
\begin{eqnarray} \label{def-asymp-system}
	x'_i(t)&=&a_i(t,x_i(t))\bigg[-\hat {b}_i(t,x_i(t))+F_i\left(\dst\sum_{p=1}^{P}\sum_{j,l=1}^{n}\hat{c}_{ijlp}(t)h_{ijlp}\big(x_j(t-\hat{\tau}_{ijp}(t)),x_l(t-\hat{\til{\tau}}_{ilp}(t))\big)\right.,\nonumber\\
& &\dst\sum_{p=1}^{P}\sum_{j,l=1}^{n}\hat{d}_{ijlp}(t)f_{ijlp}\left(\int_{-\infty}^0g_{ijp}(x_j(t+s))d\eta_{ijp}(s),\int_{-\infty}^0\til{g}_{ilp}(x_l(t+s))d{\til{\eta}}_{ilp}(s)\right)\Bigg)\nonumber\\
& &+\hat I_i(t)\bigg],\Es t\geq0,\,i=1,\ldots,n,
\end{eqnarray} 
is said to be asymptotic system of (\ref{1}) if, for each $i,j,l=1,\ldots,n$ and $p=1,\ldots,P$, the functions $\hat b_i:[0,+\infty)\times\er\to\er$,  $\hat c_{ijlp},\hat d_{ijlp},\hat I_i:[0,+\infty)\to\er$, and $\hat \tau_{ijp}, \hat{\til{\tau}}_{ilp}:[0,+\infty)\to[0,+\infty)$ are continuous such that $\hat{b}_i$ satisfies H3., i.e there is a continuous function $\hat{\beta}_i:[0,+\infty)\to[0,+\infty)$ verifying
 \begin{align}
     \label{hyp-H3^}
 	\frac{\hat{b}_i(t,u)-\hat{b}_i(t,v)}{u-v}\geq \hat{\beta}_i(t) \Es t\geq 0,\, u,v \in\mathbb{R}\text{ with }  u\neq v,
 \end{align}
 
 and 
\begin{eqnarray}\label{asymptotic-system}
\begin{array}{rcl}
\lim\limits_{t\to+\infty} \big(\beta_i(t)-\hat{\beta}_i(t)\big) &=& \lim\limits_{t\to+\infty} \big(b_i(t,w(t))-\hat{b}_i(t,w(t))\big)=\lim\limits_{t\to+\infty}\big(c_{ijlp}(t)-\hat{c}_{ijlp}(t)\big)\\
\\
&=&\lim\limits_{t\to+\infty}\big(d_{ijlp}(t)-\hat{d}_{ijlp}(t)\big)=\lim\limits_{t\to+\infty}\big(\tau_{ijp}(t)-\hat{\tau}_{ijp}(t)\big)\\
\\
 &=&\lim\limits_{t\to+\infty}\big(I_{i}(t)-\hat{I}_{i}(t)\big)=\lim\limits_{t\to+\infty}\big(\til \tau_{ilp}(t)-\hat{\til \tau}_{ilp}(t)\big)=0,
\end{array}
\end{eqnarray}
for every bounded continuous function $w:\er\to\er$.
\end{definicao}

Denoting, for all $i=1,\ldots,n$, $t\geq0$, and $\varphi=(\varphi_1,\ldots,\varphi_n)\in BC$,
\begin{eqnarray} \label{def-hat-U}
		\begin{array}{rcl}
			\hat{\mathcal{U}}_i(t,\varphi)&=&\dst\sum_{p=1}^{P}\sum_{j,l=1}^{n}\hat{c}_{ijlp}(t)h_{ijlp}\big(\varphi_j(-\hat{\tau}_{ijp}(t)),\varphi_l(-\hat{{\tau}}_{ilp}(t))\big),\\
\hat{\mathcal{V}}_i(t,\varphi)	&=&\dst \sum_{p=1}^{P}\sum_{j,l=1}^{n}\hat{d}_{ijlp}(t)f_{ijlp}\left(\int_{-\infty}^0g_{ijp}(\varphi_j(s))d\eta_{ijp}(s),\int_{-\infty}^0{g}_{ilp}(\varphi_l(s))d{\til\eta}_{ilp}(s)\right),\Es
\end{array}
\end{eqnarray} 
system \eqref{def-asymp-system} can be written as
$$
x'_i(t)=a_i(t,x_i(t))\big[-\hat{b}_i(t,x_i(t))+F_i\left(\hat{\mathcal{U}}_i(t,x_t),\hat{\mathcal{V}}_i(t,x_t)\right)+\hat{I}_i(t)\big],\Es t\geq0,\,i=1,\ldots,n.
$$
By \eqref{asymptotic-system}, it is obvious that the hypothesis $H7.$ is equivalent to the following
\begin{eqnarray}
	\limsup_{t\rightarrow +\infty} \bigg[-\un{a}_i\big(\hat{\beta}_i(t)+A_i(t)\big)+\lefteqn{\sum_{p=1}^{P}\sum_{j,l=1}^{n}\Bigg(\zeta_i|\hat{c}_{ijlp}(t)|\left(\ov{a}_j\frac{d_j}{d_i}\gamma_{ijlp}^{(1)}+\ov{a}_l\frac{d_l}{d_i}\gamma_{ijlp}^{(2)}\right)}\nonumber
		\\
		+&\dst\varsigma_i|\hat{d}_{ijlp}(t)|\left(\ov{a}_j\dst\frac{d_j}{d_i}\xi_{ijp}\mu_{ijlp}^{(1)}+\ov{a}_l\frac{d_l}{d_i}\til{\xi}_{ilp}\mu_{ijlp}^{(2)}\right)\Bigg)\bigg]<0.\nonumber
	\end{eqnarray}
\section{Global convergence}
Before considering the global convergence of the models, we show that all solutions of \eqref{1}, with bounded initial condition \eqref{1-IC}, are defined on $\er$.
\begin{lema}\label{lemma_1}
Assume H2.-H3. and H5.-H6.. Then a solution $x(t)$ of the initial value problem \eqref{1}-\eqref{1-IC} is defined on $\mathbb{R}$.
\end{lema}

\begin{proof}
Let $x(t)=(x_1(t),\dots,x_n(t))$ be a maximal (noncontinuable) solution of the initial value problem \eqref{1}-\eqref{1-IC}. By the continuation theorem \cite[Theorem 2.3]{hale1978phase}, the solution $x(t)$ is defined on $(-\infty,t^*)$ with $t^*\in(t_0,+\infty]$ and there is an increasing real sequence $(t_k)_{k\in\en}$ such that
\begin{equation}\label{lim t_k}
    \dst\lim\limits_{k\to +\infty}t_k=t^*
\end{equation}
and  $\max\left\{t^*,\lim\limits_{k\to + \infty}|x(t_k)|\right\}=+\infty$. 
To get a contradiction, we assume that $t^*\neq +\infty$. Thus 
\begin{align}\label{eq:proof-sol-def-IR}
\lim\limits_{k\to +\infty}|x(t_k)|=+\infty.
\end{align}

Define $z(t)=(z_1(t),\dots,z_n(t))=(|x_1(t)|,\dots,|x_n(t)|)$. For each  $i=1,\dots,n$, we have
$$
	z_i'(t)=sign(x_i(t))x_i'(t)=sign(x_i(t))a_i(t,x_i(t))\left[-b_i(t,x_i(t))+F_i\left(\mathcal{U}_i(t,x_t),\mathcal{V}_i(t,x_t)\right)+I_i(t)\right],
$$
for almost every $t\geq t_0$. For all $t\geq t_0$, and  integrating over $[t_0,t]$, we obtain
$$
\begin{array}{rcl}
	z_i(t)&\leq&z_i(t_0)+\dst\int_{t_0}^tz'_i(v)dv\\
	&\leq&z_i(t_0)-\dst\int_{t_0}^ta_i(v,x_i(v))sign(x_i(v))b_i(v,x_i(v))dv+\int_{t_0}^ta_i(v,x_i(v))|I_i(v)|dv\\
	& &+\dst\int_{t_0}^ta_i(v,x_i(v)\big|F_i\left(\mathcal{U}_i(v,x_v),\mathcal{V}_i(v,x_v)\right)\big|dv,\\
\end{array}
$$
and by  H2., H3., H5., and H6., we obtain
\begin{eqnarray}\label{eq:proof-sol-def-IR-A}
	z_i(t)&\leq&\|\phi\|-\dst\int_{t_0}^t\underline{a}_i\beta_i(v)z_i(v)dv+\int_{t_0}^t\ov{a}_i\big(|b_i(v,0)|+|I_i(v)|+|F_i(\mathcal{U}_i(v,0),\mathcal{V}_i(v,0))|\big)dv\nonumber\\
	&
	&	
	+\int_{t_0}^t\ov{a}_i|F_i(\mathcal{U}_i(v,x_v),\mathcal{V}_i(v,x_v))-F_i(\mathcal{U}_i(v,0),\mathcal{V}_i(v,0))|dv\nonumber\\
         & \leq &    \|\phi\|-\dst\int_{t_0}^t\underline{a}_i\beta_i(v)z_i(v)dv+\int_{t_0}^t\ov{a}_i\big(|b_i(v,0)|+|I_i(v)|+|F_i(\mathcal{U}_i(v,0),\mathcal{V}_i(v,0))|\big)dv\nonumber\\
         &
         &
         +\dst\int_{t_0}^t\ov a_i\bigg(\zeta_i|\mathcal{U}_i(v,x_v)-\mathcal{U}_i(v,0)|+\varsigma_i|\mathcal{V}_i(v,x_v)-\mathcal{V}_i(v,0)|\bigg)dv\nonumber\\
	&\leq &\|\phi\|-\dst\int_{t_0}^t\underline{a}_i\beta_i(v)z_i(v)dv+\int_{t_0}^t\ov{a}_i\big(|b_i(v,0)|+|I_i(v)|+|F_i(\mathcal{U}_i(v,0),\mathcal{V}_i(v,0))|\big)dv\nonumber\\
	& &+
	\dst\sum_{p=1}^{P}\sum_{j,l=1}^{n}\ov{a}_i\bigg(\int_{t_0}^t\zeta_i|c_{ijlp}(v)|\left|h_{ijlp}\big(x_j(v-\tau_{ijp}(v)),x_l(v-\til{\tau}_{ilp}(v))\big)-h_{ijlp}(0,0)\right|dv\nonumber\\
	& &+
	\dst\int_{t_0}^t\varsigma_i|d_{ijlp}(v)|\bigg|f_{ijlp}\left(\dst\int_{-\infty}^0g_{ijp}(x_j(v+s))d\eta_{ijp}(s),\dst\int_{-\infty}^0{g}_{ilp}(x_l(v+s))d{\til{\eta}}_{ilp}(s)\right)\nonumber\\
	& &- f_{ijlp}\big(g_{ijp}(0),\til{g}_{ilp}(0)\big)\bigg|\bigg)dv\nonumber\\
	&\leq&\|\phi\|+\dst\int_{t_0}^t\ov{a}_i\big(|b_i(v,0)|+|I_i(v)|+|F_i\left(\mathcal{U}_i(v,0),\mathcal{V}_i(v,0)\right)|\big)dv\nonumber\\
	& &+\ov{a}_i\dst\int_{t_0}^t\sum_{p=1}^{P}\sum_{j,l=1}^{n}\left(\zeta_i|c_{ijlp}(v)|\big(\gamma^{(1)}_{ijlp}+\gamma^{(2)}_{ijlp}\big)\right.\nonumber\\
	& &\left.+\varsigma_i|d_{ijlp}(v)|\big(\mu_{ijlp}^{(1)}\xi_{ijp}+\mu_{ijlp}^{(2)}{\til{\xi}}_{ilp}\big)\right)\|z_v\|dv.
\end{eqnarray}
Defining the continuous functions $\mathcal{A},\mathcal{B}:[0,+\infty)\to(0,+\infty)$  by 
$$ 
\mathcal{A}(t)=\dst\max_{i}\bigg\{\ov a_i\big(|b_i(t,0)|+|I_i(t)|+|F_i\left(\mathcal{U}_i(t,0),\mathcal{V}_i(t,0)\right)|\big)\bigg\}
$$ 
and 
$$
\mathcal{B}(t)=\dst\max_{i}\left\{\ov{a}_i\sum_{p=1}^P\sum_{j,l=1}^{n}\left(\zeta_i|c_{ijlp}(t)|\big(\gamma^{(1)}_{ijlp}+\gamma^{(2)}_{ijlp}\big)+\varsigma_i|d_{ijlp}(t)|\big(\mu_{ijlp}^{(1)}\xi_{ijp}+\mu_{ijlp}^{(2)}\til{{\xi}}_{ilp}\big)\right)\right\},
$$
respectively, from \eqref{eq:proof-sol-def-IR-A}, we obtain
$$
z_i(t)\leq\|\phi\|+\int_{t_0}^t\big(\mathcal{A}(v)+\mathcal{B}(v)\|z_v\|\big)dv,\Es i=1,\dots,n,\,t\geq t_0,
$$
thus
$$
\|z_t\|\leq\|\phi\|+\int_{t_0}^t\mathcal{A}(v)dv+\int_{t_0}^t\mathcal{B}(v)\|z_v\|dv,\Es t\geq t_0.
$$

Now, by using the generalized Gronwall's inequality \cite{hale1980ordinary}, we obtain
$$
\|z_t\|\leq \|\phi\|+\int_{t_0}^t\mathcal{A}(v)dv+\int_{t_0}^{t}\mathcal{B}(v)\left(\|\phi\|+\int_{t_0}^v\mathcal{A}(r)dr\right)\e^{\int_{v}^{t}\mathcal{B}(r)dr}dv,
$$
and from \eqref{lim t_k},
$\lim\limits_{k} |x(t_k)|=\lim\limits_{k} z(t_k)\leq \lim\limits_{k}\|z_{t_k}\|<+\infty$ which contradicts \eqref{eq:proof-sol-def-IR}.
\end{proof}

Now, we prove the boundedness of solutions of the initial value problem \eqref{1}-\eqref{1-IC}.
\begin{lema}\label{teo:bounded-solutions}
     Assume H1.-H3. and H5.-H7.. Then, solutions of the initial value problem \eqref{1}-\eqref{1-IC} are bounded.
\end{lema}
\begin{proof}
	 Let $t_0\geq0$ and $\phi\in BC$ and consider $x(t)=(x_1(t),\dots,x_n(t))$ a solution of the initial value problem \eqref{1}-\eqref{1-IC}.
	
     By H1., the functions $b_i(\cdot,0)$, $c_{ijlp}$, $d_{ijlp}$, and $I_i$ are bounded, thus there exists $M>0$ such that
     \begin{align}\label{1.1}
     	M\geq d_i^{-1}\bigg(|F_i\left(\mathcal{U}_i(t,0),\mathcal{V}_i(t,0)\right)|+|b_i(t,0)|+|I_i(t)|\bigg),\Es t\geq 0,\, i=1,\dots,n.
     \end{align}
     
     From H7., there exist $T>t_0$ and a negative number $B$ such that 
    \begin{eqnarray}\label{proof-bounded-sol-1}
    -\un{a}_i\big({\beta}_i(t)+A_i(t)\big)+\lefteqn{\sum_{p=1}^{P}\sum_{j,l=1}^{n}\Bigg(\zeta_i|{c}_{ijlp}(t)|\bigg(\ov{a}_j\frac{d_j}{d_i}\gamma_{ijlp}^{(1)}+\ov{a}_l\frac{d_l}{d_i}\gamma_{ijlp}^{(2)}\bigg)}\nonumber
		\\
		+&\dst\varsigma_i|{d}_{ijlp}(t)|\bigg(\ov{a}_j\dst\frac{d_j}{d_i}\xi_{ijp}\mu_{ijlp}^{(1)}+\ov{a}_l\frac{d_l}{d_i}\til{\xi}_{ilp}\mu_{ijlp}^{(2)}\bigg)\Bigg)<B,\Es
    \end{eqnarray}
     for $t\geq T$ and $i=1,\ldots,n$.
     
    Now, define $z(t)=(z_1(t),\ldots,z_n(t))$ by
     \begin{eqnarray}\label{def-z}
       z_i(t)=sign(x_i(t))d_i^{-1}\int_0^{x_i(t)}\frac{1}{a^*_i(t,v)}dv,\Es t\in\er,\, i=1,\ldots,n,
     \end{eqnarray}
     where, for all $v\in\er$,
     $$
       a_i^*(t,v)=\left\{\begin{array}{ll}
       	a_i(t,v),&\text{ if }t\geq t_0\\
       	a_i(t_0,v),&\text{ if }t< t_0
       	\end{array}\right..
     $$ 
     From \eqref{1} and \eqref{def-U}, we obtain, for $t> t_0$,
     \begin{eqnarray}
     	z'_i(t)&=&sign(x_i(t))d_i^{-1}\left(\frac{x_i'(t)}{a_i(t,x_i(t))}+\int_0^{x_i(t)}-\frac{\frac{\partial}{\partial t} a_i(t,v)}{a_i(t,v)^2}dv\right)\nonumber\\
     	&=&sign(x_i(t))d_i^{-1}\left(-b_i(t,x_i(t))+F_i\left(\mathcal{U}_i(t,x_t),\mathcal{V}_i(t,x_t)\right)+I_i(t)+\int_0^{x_i(t)}-\frac{\frac{\partial}{\partial t} a_i(t,v)}{a_i(t,v)^2}dv\right)\nonumber\\
     	&\leq&d_i^{-1}\left(-sign(x_i(t))(b_i(t,x_i(t))-b_i(t,0))-sign(x_i(t))\int_0^{x_i(t)}\frac{\frac{\partial}{\partial t} a_i(t,v)}{a_i(t,v)^2}dv\right.\nonumber\\
     	& &+|F_i(\mathcal{U}_i(t,x_t),\mathcal{V}_i(t,x_t))-F_i(\mathcal{U}_i(t,0),\mathcal{V}_i(t,0))|+|b_i(t,0)|+|F_i(\mathcal{U}_i(t,0),\mathcal{V}_i(t,0))|+|I_i(t)|\bigg).\nonumber
     \end{eqnarray}
     By H2., H3., H5., H6., and \eqref{def-U}, for all $t> t_0$ and $i=1,\ldots,n$, we obtain
     \begin{eqnarray}\label{proof-bounded-solutions1}
     	z'_i(t)
     	&\leq&d_i^{-1}\bigg(-(\beta_i(t)+A_i(t))|x_i(t)|+|F_i(\mathcal{U}_i(t,0),\mathcal{V}_i(t,0))|+|b_i(t,0)|+|I_i(t)|\nonumber\\
     	& & +\zeta_i|\mathcal{U}_i(t,x_t)-\mathcal{U}_i(t,0)|+\varsigma_i|\mathcal{V}_i(t,x_t)-\mathcal{V}_i(t,0)|\bigg)\nonumber\\
     	&\leq&-(\beta_i(t)+A_i(t))d_i^{-1}|x_i(t)|+d_i^{-1}\bigg(|F_i(\mathcal{U}_i(t,0),\mathcal{V}_i(t,0))|+|b_i(t,0)|+|I_i(t)|\bigg)\nonumber\\
     	& &+d_i^{-1}\sum_{p=1}^{P}\sum_{j,l=1}^{n}\bigg(\zeta_i|c_{ijlp}(t)|\big|h_{ijlp}(x_j(t-\tau_{ijp}(t)),x_l(t-\til{\tau}_{ilp}(t)))-h_{ijlp}(0,0)\big|\nonumber\\
     	& &+\varsigma_i|d_{ijlp}(t)|\bigg|f_{ijlp}\left(\int_{-\infty}^0g_{ijp}(x_j(t+s))d\eta_{ijp}(s),\int_{-\infty}^0\til{g}_{ilp}(x_l(t+s))d\til{\eta}_{ilp}(s)\right)-f_{ijlp}(0,0)\bigg|\bigg)\nonumber\\
     	&\leq&-(\beta_i(t)+A_i(t))d_i^{-1}|x_i(t)|+d_i^{-1}\bigg(|F_i(\mathcal{U}_i(t,0),\mathcal{V}_i(t,0))|+|b_i(t,0)|+|I_i(t)|\bigg)\nonumber\\
     	& &+d_i^{-1}\sum_{p=1}^{P}\sum_{j,l=1}^{n}\bigg(\zeta_i|c_{ijlp}(t)|\left(\gamma_{ijlp}^{(1)}|x_j(t-\tau_{ijp}(t))|+\gamma_{ijlp}^{(2)}|x_l(t-\til{\tau}_{ilp}(t))|\right)\nonumber\\
     	& &+\varsigma_i|d_{ijlp}(t)|\left(\mu_{ijlp}^{(1)}\xi_{ijp}\int_{-\infty}^0|x_j(t+s)|d\eta_{ijp}(s)+\mu_{ijlp}^{(2)}
    \til{\xi}_{ilp}\int_{-\infty}^0|x_l(t+s)|d\til{\eta}_{ilp}(s)\right)\bigg).\nonumber\\
     \end{eqnarray}
     
     Now, by contradiction, we assume that $x(t)$ is unbounded. From Lemma \ref{lemma_1}, $x(t)$ is defined on $\mathbb{R}$ and  from H2. and \eqref{def-z}, we conclude that there exist $i\in \{1,\dots,n\}$ and a positive real sequence $(t_k)_{k\in\mathbb{N}}$ such that $T\leq t_k  \nearrow +\infty$, $0<z_i(t_k)\nearrow+\infty$,
     \begin{align}\label{1.3}
     z_i(t_k)=\|z_{t_k}\|\geq \|z_t\|, \text{  and  } z'_i(t_k)\geq0,\Es k\in \mathbb{N},\, t\leq t_k.
     \end{align}
     From H2. and \eqref{def-z}, we also have 
     $$
       (d_j\ov{a}_j)^{-1}|x_j(t)|\leq z_j(t)\leq(d_j\un{a}_j)^{-1}|x_j(t)|,\Es t\in\er,\,j=1,\ldots,n,
     $$
     and from \eqref{1.1} and \eqref{proof-bounded-solutions1}, we obtain
     \begin{eqnarray}
       z'_i(t_k)&\leq& -(\beta_i(t_k)+A_i(t_k))\un{a}_iz_i(t_k)+\sum_{p=1}^{P}\sum_{j,l=1}^{n}\bigg(\zeta_i|c_{ijlp}(t_k)|\left(\ov{a}_j\frac{d_j}{d_i}\gamma_{ijlp}^{(1)}+\ov{a}_l\frac{d_l}{d_i}\gamma_{ijlp}^{(2)}\right)\|z_{t_k}\|\nonumber\\
       & &+\varsigma_i|d_{ijlp}(t_k)|\left(\ov{a}_j\dst\frac{d_j}{d_i}\xi_{ijp}\mu_{ijlp}^{(1)}+\ov{a}_l\frac{d_l}{d_i}\til{\xi}_{ilp}\mu_{ijlp}^{(2)}\right)\|z_k\|\bigg)+M,\Es k\in \mathbb{N},\nonumber
     \end{eqnarray}
     and consequently, from \eqref{1.3},
      \begin{eqnarray}
     	z'_i(t_k)&\leq& \left[-(\beta_i(t_k)+A_i(t_k))\un{a}_i+\sum_{p=1}^{P}\sum_{j,l=1}^{n}\bigg(\zeta_i|c_{ijlp}(t_k)|\left(\ov{a}_j\frac{d_j}{d_i}\gamma_{ijlp}^{(1)}+\ov{a}_l\frac{d_l}{d_i}\gamma_{ijlp}^{(2)}\right)\right.\nonumber\\
     	& &+\varsigma_i\left.|d_{ijlp}(t_k)|\left(\ov{a}_j\dst\frac{d_j}{d_i}\xi_{ijp}\mu_{ijlp}^{(1)}+\ov{a}_l\frac{d_l}{d_i}\til{\xi}_{ilp}\mu_{ijlp}^{(2)}\right)\bigg)\right]z_i(t_k)+M,\Es k\in \mathbb{N}.\nonumber
     \end{eqnarray}
     Hence, by \eqref{proof-bounded-sol-1},
     $$
       z'_i(t_k)\leq Bz_i(t_k)+M\overset{k\to +\infty}{\longrightarrow} -\infty,
     $$
     which contradicts \eqref{1.3}.
     \end{proof}
     
     Now we are in position to prove our main result. It presents sufficient conditions ensuring the global attractivity of solutions of system \eqref{1} and of its asymptotic systems \eqref{def-asymp-system}.

\begin{teorema}\label{convegrence of solutions theorem}
Assume H1.-H7. hold.
Then 
 $$
  \lim\limits_{t \to +\infty}|x(t)-\hat{x}(t)|=0,
 $$
for all $x(t)$ and $\hat{x}(t)$  solutions of systems \eqref{1} and \eqref{def-asymp-system} respectively, with bounded initial conditions.
\end{teorema}
\begin{proof}
Let $x(t)=(x_1(t),\dots,x_n(t))$ and $\hat{x}(t)=(\hat{x}_1(t),\dots,\hat{x}_n(t))$ be solutions of systems (\ref{1}) and (\ref{def-asymp-system}) respectively, with bounded initial conditions, and define $ y(t)=(y_1(t),\ldots,y_n(t))$ with
\begin{eqnarray}\label{def-y(t)}
  y_i(t)=sign(x_i(t)-\hat{x}_i(t))d_i^{-1}\int_{\hat{x}_i(t)}^{x_i(t)}\frac{1}{a_i(t,v)}dv.
\end{eqnarray}
From Lemma \ref{lemma_1} and Lemma \ref{teo:bounded-solutions}, we know that $x(t)$ and $\hat{x}(t)$ are bounded on $\mathbb{R}$. It follows that $y(t)$ is a nonnegative bounded function on $\mathbb{R}$ and it is possible to define the limits
$$
  u_i= \limsup_{t\to+\infty} y_i(t),\Es i=1,\dots,n.
$$
Taking $u=\dst\max_{i}\{u_i\}\in[0,+\infty)$, the result is proved if we show that $u=0$.

Let $i\in\{1,\dots,n\}$ be such that $u_i=u$. As $y_i$ is a differentiable real function, by the fluctuation lemma \cite[Lemma A.1]{smith2011introduction}, there exists a positive real sequence $(t_k)_{k\in\mathbb{N}}$ such that
\begin{align}\label{2.1}
t_k \nearrow +\infty, \text{  } y_i(t_k)\to u \text{  and  } y'_{i}(t_k)\rightarrow 0 \text{ as } k\to+\infty.
\end{align}
\begin{comment}
In fact, we have two cases:
\begin{enumerate}[1.]
	 \item If $y_i(t)$ is eventually monotone, then $\lim\limits_{t\to+\infty} y_i(t)=u,$ and for large $t$, $y_i(t)$ is a differentiable, monotone and bounded real function. Hence, there exists a positive sequence $(t_k)_{k\in\mathbb{N}}$ such that $t_k \nearrow+\infty$ and $y'_{i}(t_k)\rightarrow 0$ as $k\to+\infty.$
	 \item If $y_i(t)$ is not eventually monotone, then there exists a positive sequence $(t_k)_{k\in\mathbb{N}}$ such that $t_k \nearrow+\infty$, $y'_{i}(t_k)=0$ and $y_{i}(t_k)\rightarrow u$ as $k\to+\infty$. Thus (\ref{2.1}) holds.
\end{enumerate}
\end{comment}
Assume contrarily that $u>0$.
For $\varepsilon\in(0,u)$, there is $T=T(\varepsilon)>0$ such that $|y(t)|<u+\varepsilon$, for $t\geq T$, and
\begin{eqnarray}\label{lim-integral} \max\left\{\int_{-\infty}^{-T}d\eta_{ijp}(s),\int_{-\infty}^{-T}d\til{\eta}_{ilp}(s)\right\}<\frac{\varepsilon}{\mathcal{Y}},\Es i,j,l =1,\ldots,n,\,p= 1,\ldots,P, 
\end{eqnarray}
where $\mathcal{Y}=\dst\sup_{t\in\er}|y(t)|$.

Since $\dst\lim_{t\to+\infty}(t-\tau_{ijp}(t))=\lim_{t\to+\infty}(t-\til{\tau}_{ilp}(t))=+\infty$, $\dst\lim_{t\to+\infty}(\tau_{ijp}(t)-\hat{\tau}_{ijp}(t))=\lim_{t\to+\infty}(\til\tau_{ilp}(t)-\hat{\til\tau}_{ilp}(t))=0$, $t_k \nearrow +\infty$, and $\dst\lim_ky_{i}(t_k)=u$, then there exists a natural number $k_0$ such that
\begin{equation}\label{N}
    t_k>2T,\Es t_k-\hat{\tau}_{ijp}(t_k)>T,\Es t_k-\hat{\til\tau}_{ilp}(t_k)>T \text{ and } y_i(t_k)>u-\varepsilon>0,
\end{equation}
  
for all $k\geq k_0$, $i,j,l=1,\ldots,n$, and $p=1,\ldots,P$. Consequently, for $k\geq k_0$, we have
\begin{eqnarray}\label{lim-atraso-discreto}
	|y( t_k-\hat{\tau}_{ijp}(t_k))|<u+\varepsilon\,\text{ and }\, |y( t_k-\hat{\til\tau}_{ilp}(t_k))|<u+\varepsilon.
\end{eqnarray}

For $k\geq k_0$, from systems \eqref{1}, \eqref{def-asymp-system} and, notations \eqref{def-U}, \eqref{def-hat-U},  we have  
\begin{eqnarray}
y'_i(t_k)&=&sign(x_i(t_k)-\hat{x}_i(t_k))d^{-1}_i\left[\frac{x'_i(t_k)}{a_i(t_k,x_i(t_k))}-\frac{\hat{x}'_i(t_k)}{a_i(t_k,\hat{x}_i(t_k))}+\int_{\hat{x}_i(t_k)}^{x_i(t_k)}-\frac{\frac{\partial}{\partial t}a_i(t_k,v)}{a_i(t_k,v)^2}dv\right]\nonumber\\
&=&sign(x_i(t_k)-\hat{x}_i(t_k))d^{-1}_i\bigg[-b_i(t_k,x_i(t_k))+\hat{b}_i(t_k,\hat{x}_i(t_k))+\int_{\hat{x}_i(t_k)}^{x_i(t_k)}-\frac{\frac{\partial}{\partial t}a_i(t_k,v)}{a_i(t_k,v)^2}dv\nonumber\\
& &+F_i(\mathcal{U}_i(t_k,x_{t_k}),\mathcal{V}_i(t_k,x_{t_k}))-F_i(\hat{\mathcal{U}}_i(t_k,\hat{x}_{t_k}),\hat{\mathcal{V}}_i(t_k,\hat{x}_{t_k}))+I_i(t_k)-\hat{I}_i(t_k)\bigg]\nonumber\\
&\leq&sign(x_i(t_k)-\hat{x}_i(t_k))d^{-1}_i\left[-b_i(t_k,x_i(t_k))+\hat{b}_i(t_k,\hat{x}_i(t_k))+\int_{\hat{x}_i(t_k)}^{x_i(t_k)}-\frac{\frac{\partial}{\partial t}a_i(t_k,v)}{a_i(t_k,v)^2}dv\right]\nonumber\\
& &+d_i^{-1}|F_i(\mathcal{U}_i(t_k,x_{t_k}),\mathcal{V}_i(t_k,x_{t_k}))-F_i(\hat{\mathcal{U}}_i(t_k,\hat{x}_{t_k}),\hat{\mathcal{V}}_i(t_k,\hat{x}_{t_k})|+d_i^{-1}|I_i(t_k)-\hat{I}_i(t_k)|.\nonumber
\end{eqnarray}
From H2., H3., and H6., we obtain
\begin{eqnarray}
	y_i'(t_k)&\leq& -sign(x_i(t_k)-\hat{x}_i(t_k))d^{-1}_i\left[\big(b_i(t_k,x_i(t_k))-b_i(t_k,\hat{x}_i(t_k))\big)+\int_{\hat{x}_i(t_k)}^{x_i(t_k)}\frac{\frac{\partial}{\partial t}a_i(t_k,v)}{a_i(t_k,v)^2}dv\right]\nonumber\\
	& &+d^{-1}_i\big|\hat{b}_i(t_k,\hat{x}_i(t_k))-b_i(t_k,\hat{x}_i(t_k))\big|+d^{-1}_i\big|I_i(t_k)-\hat{I}_i(t_k)\big|\nonumber\\
	& &+d^{-1}_i\bigg(\zeta_i|\mathcal{U}_i(t_k,x_{t_k})-\hat{\mathcal{U}}_i(t_k,\hat{x}_{t_k})|+\varsigma_i|\mathcal{V}_i(t_k,x_{t_k})-\hat{\mathcal{V}}_i(t_k,\hat{x}_{t_k})|\bigg)\nonumber\\
	&\leq& -d^{-1}_i\big(\beta_i(t_k)+A_i(t_k)\big)|x_i(t_k)-\hat{x}_i(t_k)|+d^{-1}_i\big|\hat{b}_i(t_k,\hat{x}_i(t_k))-b_i(t_k,\hat{x}_i(t_k))\big|\nonumber\\
	& &+d^{-1}_i\big|I_i(t_k)-\hat{I}_i(t_k)\big|+d^{-1}_i\bigg(\zeta_i|\mathcal{U}_i(t_k,x_{t_k})-\hat{\mathcal{U}}_i(t_k,\hat{x}_{t_k})|+\varsigma_i|\mathcal{V}_i(t_k,x_{t_k})-\hat{\mathcal{V}}_i(t_k,\hat{x}_{t_k})|\bigg).\nonumber
\end{eqnarray}
From \eqref{def-U} and \eqref{def-hat-U}, we have

\begin{eqnarray}
         y_i'(t_k)
       &\leq& -d^{-1}_i\big(\beta_i(t_k)+A_i(t_k)\big)|x_i(t_k)-\hat{x}_i(t_k)|\nonumber\\
       & +&d^{-1}_i\big|\hat{b}_i(t_k,\hat{x}_i(t_k))-b_i(t_k,\hat{x}_i(t_k))\big|+d^{-1}_i\big|I_i(t_k)-\hat{I}_i(t_k)\big|\nonumber\\
       & +&d_i^{-1}\sum_{p=1}^{P}\sum_{j,l=1}^{n}\Bigg[\zeta_i\Bigg(\big|c_{ijlp}(t_k)-\hat{c}_{ijlp}(t_k)\big|\big|h_{ijlp}(\hat{x}_j(t_k-\hat{\tau}_{ijp}(t_k)),\hat{x}_l(t_k-\hat{\til\tau}_{ilp}(t_k)))\big|\nonumber\\
       & +&|c_{ijlp}(t_k)|\bigg(\big|h_{ijlp}(x_j(t_k-\hat\tau_{ijp}(t_k)),x_l(t_k-\hat{\til{\tau}}_{ilp}(t_k)))-h_{ijlp}(\hat{x}_j(t_k-\hat\tau_{ijp}(t_k)),\hat{x}_l(t_k-\hat{\til{\tau}}_{ilp}(t_k)))\big|\nonumber\\
       &+&\big|h_{ijlp}(x_j(t_k-\tau_{ijp}(t_k)),x_l(t_k-\til{\tau}_{ilp}(t_k)))-h_{ijlp}(x_j(t_k-\hat\tau_{ijp}(t_k)),x_l(t_k-\hat{\til{\tau}}_{ilp}(t_k)))\big|\bigg)\Bigg)\nonumber\\
       & +&\varsigma_i\Bigg(|d_{ijlp}(t_k)-\hat{d}_{ijlp}(t_k)\big|\left|f_{ijlp}\left(\int_{-\infty}^0g_{ijp}(\hat{x}_j(t_k+s)d\eta_{ijp}(s),\int_{-\infty}^0\til{g}_{ilp}(\hat{x}_l(t_k+s))d\til{\eta}_{ilp}(s)\right)\right|\nonumber\\
       &
       &+|d_{ijlp}(t_k)|\bigg|f_{ijlp}\left(\int_{-\infty}^0g_{ijp}(x_j(t_k+s)d\eta_{ijp}(s),\int_{-\infty}^0\til{g}_{ilp}(x_l(t_k+s))d\til{\eta}_{ilp}(s)\right)\nonumber\\
       &
       &-f_{ijlp}\left(\int_{-\infty}^0g_{ijp}(\hat{x}_j(t_k+s)d\eta_{ijp}(s),\int_{-\infty}^0\til{g}_{ilp}(\hat{x}_l(t_k+s))d\til{\eta}_{ilp}(s)\right)\bigg|\Bigg)\Bigg],\nonumber
\end{eqnarray}
and by hypotheses H5. and H6., we obtain
\begin{eqnarray}
	y'_i(t_k)&\leq& -d_i^{-1}(\beta_i(t_k)+A_i(t_k))|x_i(t_k)-\hat{x}_i(t_k)|\nonumber\\
	& &+d_i^{-1}\left|\hat{b}_i(t_k,\hat{x}_i(t_k))-b_i(t_k,\hat{x}_i(t_k))\right|+d^{-1}_i\left|I_i(t_k)-\hat{I}_i(t_k)\right|\nonumber\\
	& &+d_i^{-1}\sum_{p=1}^{P}\sum_{j,l=1}^{n}\Bigg[\zeta_i\bigg(\big|c_{ijlp}(t_k)-\hat{c}_{ijlp}(t_k)\big|\big|h_{ijlp}(\hat{x}_j(t_k-\hat{\tau}_{ijp}(t_k)),\hat{x}_l(t_k-\hat{\til{\tau}}_{ilp}(t_k)))\big|\nonumber\\
	& & + |c_{ijlp}(t_k)|\left(\gamma_{ijlp}^{(1)}|x_j(t_k-\hat\tau_{ijp}(t_k))-\hat{x}_j(t_k-\hat\tau_{ijp}(t_k))|\right.\nonumber\\
	&  & +\left.\gamma_{ijlp}^{(2)}|x_l(t_k-\hat{\til{\tau}}_{ilp}(t_k))-\hat{x}_l(t_k-\hat{\til{\tau}}_{ilp}(t_k))\big|\right)\nonumber\\
	& &
	+|c_{ijlp}(t_k)|\left(\gamma_{ijlp}^{(1)}|x_j(t_k-\tau_{ijp}(t_k))-x_j(t_k-\hat\tau_{ijp}(t_k))|\right.\nonumber\\
	& &+\left.\gamma_{ijlp}^{(2)}|x_l(t_k-\til{\tau}_{ilp}(t_k))-x_l(t_k-\hat{\til{\tau}}_{ilp}(t_k))\big|\right)\bigg)\nonumber\\
	&
	&
	+\varsigma_i\Bigg(\big|d_{ijlp}(t_k)-\hat{d}_{ijlp}(t_k)\big|\left|f_{ijlp}\left(\int_{-\infty}^0g_{ijp}(\hat{x}_j(t_k+s)d\eta_{ijp}(s),\int_{-\infty}^0\til{g}_{ilp}(\hat{x}_l(t_k+s))d\til{\eta}_{ilp}(s)\right)\right|\nonumber\\
	&
	&+|d_{ijlp}(t_k)|\left(\mu_{ijlp}^{(1)}\xi_{ijp}\int_{-\infty}^0|x_j(t_k+s)-\hat{x}_j(t_k+s)|d\eta_{ijp}(s)\right.\nonumber\\
	& &
	+\left.\mu_{ijlp}^{(2)}\til{\xi}_{ilp}\int_{-\infty}^0|x_l(t_k+s)-\hat{x}_l(t_k+s)|d\til{\eta}_{ilp}(s)\right)\Bigg)\Bigg].\nonumber
\end{eqnarray}
 From H2. and \eqref{def-y(t)}, we have
 $$
   (d_i\ov{a}_i)^{-1}|x_i(t)-\hat{x}_i(t)|\leq y_i(t)\leq(d_i\un{a}_i)^{-1}|x_i(t)-\hat{x}_i(t)|,\Es t\geq0,\,i=1,\ldots,n,
 $$
 and consequently,
\begin{eqnarray}\label{2.3}
y'_i(t_k)&\leq& -(\beta_i(t_k)+A_i(t_k))\un{a}_iy_i(t_k)+\delta_i(t_k)\nonumber\\
& &+\sum_{p=1}^{P}\sum_{j,l=1}^{n}\bigg[\zeta_i|c_{ijlp}(t_k)|\left(\gamma_{ijlp}^{(1)}\ov{a}_j\frac{d_j}{d_i}y_j(t_k-\hat{\tau}_{ijp}(t_k))+\gamma_{ijlp}^{(2)}\ov{a}_l\frac{d_l}{d_i}y_l(t_k-\hat{\til{\tau}}_{ilp}(t_k))\right)\nonumber\\
&
&
+\varsigma_i|d_{ijlp}(t_k)|\left(\mu_{ijlp}^{(1)}\xi_{ijp}\ov{a}_j\frac{d_j}{d_i}\int_{-\infty}^0y_j(t_k+s)d\eta_{ijp}(s)+\mu_{ijlp}^{(2)}\til{\xi}_{ilp}\ov{a}_l\frac{d_l}{d_i}\int_{-\infty}^0y_l(t_k+s)d\til{\eta}_{ilp}(s)\right)\bigg],\nonumber\\
\end{eqnarray}
where 
\begin{eqnarray}
\delta_{i}(t)&=&d_i^{-1}\left|\hat{b}_i(t,\hat{x}_i(t))-b_i(t,\hat{x}_i(t))\right|+d^{-1}_i\left|I_i(t)-\hat{I}_i(t)\right|\nonumber\\
& &+d^{-1}_i\sum_{p=1}^{P}\sum_{j,l=1}^{n}\Bigg[\zeta_i\bigg(\big|c_{ijlp}(t)-\hat{c}_{ijlp}(t)\big|\big|h_{ijlp}(\hat{x}_j(t-\hat{\tau}_{ijp}(t)),\hat{x}_l(t-\hat{\til{\tau}}_{ilp}(t)))\big|\nonumber\\
& &	+|c_{ijlp}(t)|\left(\gamma_{ijlp}^{(1)}|x_j(t-\tau_{ijp}(t))-x_j(t-\hat\tau_{ijp}(t))|\right.+\left.\gamma_{ijlp}^{(2)}|x_l(t-\til{\tau}_{ilp}(t))-x_l(t-\hat{\til{\tau}}_{ilp}(t))\big|\right)\nonumber\\
& &
	+\varsigma_i\big|d_{ijlp}(t)-\hat{d}_{ijlp}(t)\big|\left|f_{ijlp}\left(\int_{-\infty}^0g_{ijp}(\hat{x}_j(t+s)d\eta_{ijp}(s),\int_{-\infty}^0\til{g}_{ilp}(\hat{x}_l(t+s))d\til{\eta}_{ilp}(s)\right)\right|\Bigg].\nonumber
\end{eqnarray}

By H1., $t\mapsto b_j(t,w)$ is bounded for all $w\in\er$ and $j=1,\ldots,n$, and by H3., $w\mapsto b_j(t,w)$ is a non-decreasing function for all $t\in[0,+\infty)$. As $x_j(t)$ is bounded, then $t\mapsto b_j(t,x_j(t))$ is a bounded function. Moreover from H1., \eqref{1}, and the continuity of $F_i$, $h_{ijlp}$, $f_{ijlp}$, $g_{ijp}$, $\til g_{ilp}$, we obtain that $x'_j(t)$ is a bounded function on $(0,+\infty)$, thus $x_j(t)$ is uniformly continuous for all $j=1,\ldots,n$. Finally, as $\hat{x}(t)$ is bounded, by H5. and from \eqref{asymptotic-system}, we conclude that 
\begin{eqnarray}\label{24}
	\lim_{t\to+\infty}\delta_{i}(t)=0.
\end{eqnarray}
From \eqref{N}, \eqref{lim-atraso-discreto} and \eqref{2.3}, for all $k\geq k_0$, we have
\begin{eqnarray*}\label{eq24a}
	y'_i(t_k)&\leq& -(\beta_i(t_k)+A_i(t_k))\un{a}_i(u-\varepsilon)+\delta_i(t_k),\nonumber\\
	& &+\sum_{p=1}^{P}\sum_{j,l=1}^{n}\bigg[\zeta_i|c_{ijlp}(t_k)|\left(\gamma_{ijlp}^{(1)}\ov{a}_j\frac{d_j}{d_i}+\gamma_{ijlp}^{(2)}\ov{a}_l\frac{d_l}{d_i}\right)(u+\varepsilon)\nonumber\\
	&
	&
	+\varsigma_i|d_{ijlp}(t_k)|\mu_{ijlp}^{(1)}\xi_{ijp}\ov{a}_j\frac{d_j}{d_i}\left(\int_{-\infty}^{-T}y_j(t_k+s)d\eta_{ijp}(s)+\int_{-T}^0y_j(t_k+s)d\eta_{ijp}(s)\right)\nonumber\\
	& &
	+\varsigma_i|d_{ijlp}(t_k)|\mu_{ijlp}^{(2)}\til{\xi}_{ilp}\ov{a}_l\frac{d_l}{d_i}\left(\int_{-\infty}^{-T}y_l(t_k+s)d\til{\eta}_{ilp}(s)+\int_{-T}^0y_l(t_k+s)d\til{\eta}_{ilp}(s)\right)\bigg]\nonumber,\\
\end{eqnarray*}
and from \eqref{lim-integral}, we obtain
\begin{eqnarray}\label{eq24a-2}
	y'_i(t_k)&\leq& -(\beta_i(t_k)+A_i(t_k))\un{a}_i(u-\varepsilon)+\delta_i(t_k)\nonumber\\
	& &+\sum_{p=1}^{P}\sum_{j,l=1}^{n}\bigg[\zeta_i|c_{ijlp}(t_k)|\left(\gamma_{ijlp}^{(1)}\ov{a}_j\frac{d_j}{d_i}+\gamma_{ijlp}^{(2)}\ov{a}_l\frac{d_l}{d_i}\right)(u+\varepsilon)\nonumber\\
	&
	&
	+\varsigma_i|d_{ijlp}(t_k)|\mu_{ijlp}^{(1)}\xi_{ijp}\ov a_j\frac{d_j}{d_i}\left(\varepsilon+\int_{-T}^0y_j(t_k+s)d\eta_{ijp}(s)\right)\nonumber\\
	& &
	+\varsigma_i|d_{ijlp}(t_k)|\mu_{ijlp}^{(2)}{\til\xi}_{ilp}\ov a_l\frac{d_l}{d_i}\left(\varepsilon+\int_{-T}^0y_l(t_k+s)d\til{\eta}_{ilp}(s)\right)\bigg],\nonumber
\end{eqnarray}
and finally, as $t_k>2T$, $|y(t)|\leq u+\varepsilon$ for $t\geq T$, and $\int_{-\infty}^0d\eta_{ijp}(s)=\int_{-\infty}^0d\til\eta_{ilp}(s)=1$, we obtain
\begin{eqnarray}\label{eq24a-3}
	y'_i(t_k)&\leq& -(\beta_i(t_k)+A_i(t_k))\un{a}_i(u-\varepsilon)+\delta_i(t_k)\nonumber\\
	& &+\sum_{p=1}^{P}\sum_{j,l=1}^{n}\bigg[\zeta_i|c_{ijlp}(t_k)|\left(\gamma_{ijlp}^{(1)}\ov{a}_j\frac{d_j}{d_i}+\gamma_{ijlp}^{(2)}\ov{a}_l\frac{d_l}{d_i}\right)(u+\varepsilon)\nonumber\\
	&
	&
	+\varsigma_i|d_{ijlp}(t_k)|\left(\mu_{ijlp}^{(1)}\xi_{ijp}\ov{a}_j\frac{d_j}{d_i}+\mu_{ijlp}^{(2)}\til{\xi}_{ilp}\ov{a}_l\frac{d_l}{d_i}\right)(u+2\varepsilon)\bigg].
\end{eqnarray}
Since $\dst\lim_{k\to +\infty} y_i'(t_k)=0$, then by letting $\varepsilon\rightarrow 0^+$ and $k\rightarrow +\infty$, it follows from \eqref{2.1}, \eqref{24}, \eqref{eq24a-3}, and hypothesis H7. that
\begin{eqnarray}
	0\leq \bigg[\dst\limsup_{k\to +\infty}\bigg(-(\beta_i(t_k)+A_i(t_k))\lefteqn{\un{a}_i+\sum_{p=1}^{P}\sum_{j,l=1}^{n}\bigg(\zeta_i|c_{ijlp}(t_k)|\left(\gamma_{ijlp}^{(1)}\ov{a}_j\frac{d_j}{d_i}+\gamma_{ijlp}^{(2)}\ov{a}_l\frac{d_l}{d_i}\right)}\nonumber\\
	& +\varsigma_i|d_{ijlp}(t_k)|\left(\mu_{ijlp}^{(1)}\xi_{ijp}\ov{a}_j\dst\frac{d_j}{d_i}+\mu_{ijlp}^{(2)}\til{\xi}_{ilp}\ov{a}_l\frac{d_l}{d_i}\right)\bigg)\bigg]u<0,\nonumber
	\end{eqnarray}
which is a contradiction. Thus $u=0$.
\end{proof}

Obviously, the system (\ref{1}) can be considered as an asymptotic system of itself. Thus, we have the following result
\begin{corolario}\label{cor:criterio-si-proprio}
Assume H1.-H7. hold. 

If $x(t)$ and $\hat{x}(t)$ are solutions of systems \eqref{1} with bounded initial conditions, then 
$$
  \lim\limits_{t \to +\infty}|x(t)-\hat{x}(t)|=0.
$$
\end{corolario}

%\section{Applications}
\section{Applications}
The following examples demonstrate the efficiency of our results, and we provide a comparison with various stability criteria from the literature.

\exemplo
Consider the following low-order Cohen--Grossberg neural network model:
\begin{eqnarray}\label{eq:static-model with k}
	x_i'(t)=a_i(x_i(t))\left[-b_i(t,x_i(t))+G_i\left(\sum_{j=1}^nc_{ij}(t)\int_0^{+\infty}x_j(t-u)K_{ij}(u)du\right)\right],
\end{eqnarray}
for $t\geq0$ and $i=1,\ldots,n$, where $a_i:\er\to(0,+\infty)$, $b_i:[0,+\infty)\times\er\to\er$, $c_{ij}:[0,+\infty)\to\er$, $G_i:\er\to\er$, and $K_{ij}:[0,+\infty)\to[0,+\infty)$ are continuous functions such that \begin{eqnarray}\label{ex:kernel1}
      \int_0^{+\infty}K_{ij}(u)du=1,
\text{ for all } i,j=1,\ldots,n.
\end{eqnarray}
 
Considering $\eta_{ij}:(-\infty,0]\to\er$ defined by  $\eta_{ij}(s)=\int_{-\infty}^sK_{ij}(-v)dv,\, s\in(-\infty,0]$, for each $i,j=1,\ldots,n$, we have $\eta_{ij}$ non-decreasing such that  $\eta_{ij}(0)-\eta_{ij}(-\infty)=1$.

Therefore, the model \eqref{eq:static-model with k} can be written in the form 
    \begin{eqnarray*}\label{static-model}
 	x_i'(t)=a_i(x_i(t))\left[-b_i(t,x_i(t))+G_i\left(\sum_{j=1}^nc_{ij}(t)\int_{-\infty}^0x_j(t+s)d\eta_{ij}(s)\right)\right],
    \end{eqnarray*}
\begin{comment}
Consider the following low-order Cohen-Grossberg neural network model:
 \begin{eqnarray}\label{static-model}
 	x_i'(t)=a_i(x_i(t))\left[-b_i(t,x_i(t))+G_i\left(\sum_{j=1}^nc_{ij}(t)\int_{-\infty}^0x_j(t+s)d\eta_{ij}(s)\right)\right],
 \end{eqnarray}
 for $t\geq0$ and $i=1,\ldots,n$, where  $a_i:\er\to(0,+\infty)$, $b_i:[0,+\infty)\times\er\to\er$, $c_{ij} :[0,+\infty)\to\er$, and the non-decreasing function $\eta_{ij}:(-\infty,0]\to \er$ such that $\eta_{ij}(0)-\eta_{ij}(-\infty)=1$ is continuous function. 
 \end{comment}
which is a particular situation of \eqref{1}. Consequently, from Corollary \ref{cor:criterio-si-proprio}, we obtain the following global attractivity criterion for model \eqref{eq:static-model with k}.
 %In \cite{submitted}, the following result for the existence of a periodic solution of \eqref{static-model} was established.
 \begin{corolario}\label{cor-crit-static-model}
 	Assume that H3. and \eqref{ex:kernel1} hold and, for each $i,j=1,\ldots,n$, $c_{ij}$ is a bounded function, $G_i$ is a Lipshcitz function with Lipschitz constant $\varsigma_i>0$, and there exist $\un{a}_i,\ov{a}_i>0$ such that 
 	\begin{eqnarray}\label{lim:ais}
 	  \un{a}_i\leq a_i(u)\leq\ov{a}_i,\Es\forall u\in\er.
 	\end{eqnarray}
 	 If there exists $d=(d_1,\ldots,d_n)>0$ such that,
 	\begin{eqnarray*}
 		\limsup_{t\to+\infty}\left(\un{a}_i\beta_i(t)d_i-\ov{a}_i\sum_{j=1}^{n}\varsigma_i|c_{ij}(t)|d_j\right)>0,\Es i=1,\ldots,n,
 	\end{eqnarray*}
 	then any two solutions $x(t)$ and $\hat{x}(t)$ of \eqref{eq:static-model with k}, with bounded initial condition, verify
 	$$
 	  \lim\limits_{t \to +\infty}|x(t)-\hat{x}(t)|=0.
 	$$ 
 \end{corolario}

In case of \eqref{eq:static-model with k} being an $\omega-$periodic model, for some $\omega>0$, from \cite[Theorem 5.2]{submitted}, we obtain sufficient conditions for the existence of an $\omega-$periodic solution of \eqref{eq:static-model with k} and the following result holds.
 \begin{corolario}\label{cor-generaliza-ncube}
 	Assume that \eqref{ex:kernel1} holds and, for each $i,j=1,\ldots,n$,  $G_i$ is a Lipshcitz function with Lipschitz constant $\varsigma_i>0$, $a_i$ verifies \eqref{lim:ais},  the functions $t\to b_i(t,u)$ and $c_{ij}$ are $\omega-$periodic for all $u\in\er$, and  there exist $\omega-$periodic continuous functions $\beta_i,\beta_i^*:[0,+\infty)\to(0,+\infty)$ such that
 	\begin{eqnarray}\label{cod-H3-para-periodic}
 		\beta_i(t)\leq\frac{b_i(t,u)-b_i(t,v)}{u-v}\leq \beta_i^*(t), \Es \forall t\in[0,\omega], \, u,v \in\mathbb{R},  u\neq v.
 	\end{eqnarray}
 
 If there exists $d=(d_1,\ldots,d_n)>0$ such that,
 \begin{eqnarray}\label{existence of periodic for static}
 \un{a}_i\beta_i(t)d_i>\ov{a}_i\sum_{j=1}^{n}\varsigma_i|c_{ij}(t)|d_j,\Es \forall t\in[0,\omega],\, i=1,\ldots,n,
 \end{eqnarray}
 then there exists an $\omega-$periodic solution of \eqref{eq:static-model with k}, $\tilde{x}(t)$, which is globally attractive in the set of the solutions of \eqref{eq:static-model with k}, $x(t)$,  with bounded initial condition, i.e.
 $$
 \lim\limits_{t \to +\infty}|x(t)-\tilde{x}(t)|=0.
 $$ 
 \end{corolario}

\begin{rem}
	In \cite{ncube2020existence}, sufficient conditions for the existence and global asymptotic stability of an equilibrium point of the following autonomous static neural network model 
	\begin{eqnarray}\label{eq:static-model-INcube}
		x_i'(t)=-x_i(t)+G_i\left(\sum_{j=1}^nc_{ij}\int_0^{+\infty}x_j(t-u)K_{ij}(u)du\right),\,t\geq0,\,i=1,\ldots,n,
	\end{eqnarray}
	were established. We stress that model \eqref{eq:static-model-INcube} is a particular situation of \eqref{eq:static-model with k}, and assumptions in Corollary \ref{cor-generaliza-ncube} are weaker than the ones in \cite[Theorem 3.2]{ncube2020existence}. In fact, conditions \eqref{lim:ais} and \eqref{cod-H3-para-periodic} trivially hold in model \eqref{eq:static-model-INcube}, for each $i,j=1,\ldots,n$, they assumed that $G_i$ satisfies stronger conditions than  being Lipschitz and the kernel function, $K_{ij}$, verifies \eqref{ex:kernel1} with the additional assumptions
	$$
	\int_0^{+\infty}uK_{ij}(u)du<+\infty\Es\text{and}\Es K_{ij}=K_{ji}.
	$$
	For model  \eqref{eq:static-model-INcube}, inequality \eqref{existence of periodic for static} reads as
	$$
	d_i>\sum_{j=1}^{n}\varsigma_i|c_{ij}|d_j,\Es i=1,\ldots,n,
	$$
	which is equivalent to the matrix
	\begin{eqnarray*}
		\mathcal{N}=I_n-\big[\varsigma_i|c_{ij}|\big]_{i,j=1}^n,
	\end{eqnarray*}
	where $I_n$ denotes the $n$-dimension identity matrix, being a non-singular M-matrix (see \cite{fiedler2008special}), which is assumed in \cite[Theorem 3.2]{ncube2020existence}.
\end{rem}

\exemplo 
\label{high-order-ex}
Considering in \eqref{1} and \eqref{def-asymp-system} 
$$F_i(u,v)=G_i(u)+H_i(v),\Es \text{ for all } i=1,\ldots,n,\, \text{ and } u,v\in\er,$$
where  $G_i,H_i:\er\mapsto\er$ are continuous functions, we have the following high-order Cohen--Grossberg neural network models, 
     \begin{eqnarray} \label{model-studied in our work}
	x'_i(t)&=&a_i(t,x_i(t))\bigg[-b_i(t,x_i(t))+G_i\bigg(\dst\sum_{p=1}^{P}\sum_{j,l=1}^{n}c_{ijlp}(t)h_{ijlp}\big(x_j(t-\tau_{ijp}(t)),x_l(t-{\til{\tau}}_{ilp}(t))\big)\bigg)\nonumber\\
	& &+H_i\bigg(\dst\sum_{q=1}^{Q}\sum_{j,l=1}^{n}d_{ijlq}(t)f_{ijlq}\left(\int_{-\infty}^0g_{ijq}(x_j(t+s))d\eta_{ijq}(s),\int_{-\infty}^0{\til{g}}_{ilq}(x_l(t+s))d{\til{\eta}}_{ilq}(s)\right)\bigg)\nonumber\\
	& &+I_i(t)\bigg],\Es t\geq 0,\Es \,i=1,\ldots,n,
\end{eqnarray}
and
\begin{eqnarray} \label{ass-model-studied inour work}
	x'_i(t)&=&a_i(t,x_i(t))\bigg[-\hat{b}_i(t,x_i(t))+G_i\bigg(\dst\sum_{p=1}^{P}\sum_{j,l=1}^{n}\hat{c}_{ijlp}(t)h_{ijlp}\big(x_j(t-\hat{\tau}_{ijp}(t)),x_l(t-\hat{{\til{\tau}}}_{ilp}(t))\big)\bigg)\nonumber\\
	& &+H_i\bigg({\dst\sum_{q=1}^{Q}}\sum_{j,l=1}^{n}\hat{d}_{ijlq}(t)f_{ijlq}\left(\int_{-\infty}^0g_{ijq}(x_j(t+s))d\eta_{ijq}(s),\int_{-\infty}^0{\til{g}}_{ilq}(x_l(t+s))d{\til{\eta}}_{ilq}(s)\right)\bigg)\nonumber\\
	& &+\hat I_i(t)\bigg],\Es t\geq 0,\Es \,i=1,\ldots,n,
\end{eqnarray}
respectively.

If H3., \eqref{hyp-H3^}, and \eqref{asymptotic-system} hold, then the system \eqref{ass-model-studied inour work} is an asymptotic system of \eqref{model-studied in our work}. In \cite{submitted}, the 
 following result for the existence of a periodic solution of \eqref{ass-model-studied inour work}  was established.

\begin{teorema}\cite[Theorem 5.2]{submitted}\label{periodic-soln-exist}
    Assume H2., H5., H6., and the following hypotheses:
    \begin{enumerate}
        %\item For each $i=1,\ldots,n$, there exist $\ov a_i\, , \un a_i>0$ such that 
 % \begin{eqnarray*}      \un a_i< a_i(t,u)< \ov a_i,\Es \forall t\geq0,\, u\in\er;\end{eqnarray*}
        \item For each $i=1,\dots,n$, there exist $\zeta_i$, $\varsigma_i>0$ such that
        $$\begin{array}{cc}
          |G_i(u)-G_i(v)|\leq\zeta_i|u-v|,   & |H_i(u)-H_i(v)|\leq\varsigma_i|u-v|,\, u,v\in \er{\color{blue}{;}} \\
        \end{array}$$

        \item There is $\omega>0$ such that, for each $i,j,l=1,\ldots,n$,  $p=1,\ldots,P$, and $q=1,\ldots,Q$,
	{$$
	\begin{array}{llll}
	a_i(t,u)=a_i(t+\omega,u), & \hat c_{ijlp}(t)=\hat c_{ijlp}(t+\omega), & \hat\tau_{ijp}(t)=\hat\tau_{ijp}(t+\omega),\\
	\hat {b}_i(t,u)=\hat{b}_i(t+\omega,u), & \hat {d}_{ijlq}(t)=\hat{d}_{ijlq}(t+\omega), & \hat{\til{\tau}}_{ijp}(t)=\hat {\til{\tau}}_{ijp}(t+\omega), \text{ and }\\
   \hat I_i(t)=\hat I_i(t+\omega),& \text{ for all }\,  t\geq0 \text{ and  }
   u\in\er;
	\end{array}	 	
$$}
\item  For each $i=1,\dots,n$, there exist $\omega-$periodic continuous functions $\hat{\beta}_i,\hat{\beta}_i^*:[0,+\infty)\to(0,+\infty)$ such that
	\begin{align*}
		\hat{\beta}_i(t)\leq\frac{\hat b_i(t,u)-\hat b_i(t,v)}{u-v}\leq \hat\beta_i^*(t), \Es \forall t\in[0,\omega], \,  u,v\in \er,\, u\neq v;
	\end{align*}
\item 
	There exists $d=(d_1,\ldots,d_n)>0$ such that for all $t\in[0,\omega]$, and $i=1,\ldots,n$,	
	\begin{eqnarray}\label{4.3b}
		\hat\beta_i(t)>\sum_{j,l=1}^{n}\lefteqn{\left[{{\sum_{p=1}^{P}}}\dst\zeta_i|\hat c_{ijlp}(t)|\left(\frac{d_j}{d_i}\gamma^{(1)}_{ijlp}+\frac{d_l}{d_i}\gamma^{(2)}_{ijlp}\right)\right.}\nonumber\\
		&+\dst\left.{{\sum_{q=1}^{Q}}}\varsigma_i|\hat d_{ijlq}(t)|\left(\frac{d_j}{d_i}\mu^{(1)}_{ijlq}\xi_{ijq}+\frac{d_l}{d_i}\mu^{(2)}_{ijlq}\til{\xi}_{ilq}\right)\right].
	\end{eqnarray}
    \end{enumerate}
        Then, the system \eqref{ass-model-studied inour work} has an $\omega-$periodic solution.
\end{teorema}
From Theorems \ref{convegrence of solutions theorem} and \ref{periodic-soln-exist}, we obtain the following result:%, which gives sufficient conditions for showing that solutions of \eqref{1} converge to a periodic function.
\begin{corolario}\label{cor 4.4}
    Assume H3. and the hypotheses in Theorem \ref{periodic-soln-exist}. If \eqref{asymptotic-system} holds and there is $d=(d_1,\ldots,d_n)>0$ such that
   \begin{align}\label{cor:cond-est-modelo-AJC}
       \dst-(\hat\beta_i(t)+A_i(t))\lefteqn{\un{a}_i+\sum_{j,l=1}^{n}\left(\sum_{p=1}^{P}\zeta_i|\hat c_{ijlp}(t)|\left(\gamma_{ijlp}^{(1)}\ov{a}_j\frac{d_j}{d_i}+\gamma_{ijlp}^{(2)}\ov{a}_l\frac{d_l}{d_i}\right)\right.}\nonumber\\
	& +\left.\sum_{q=1}^{Q}\varsigma_i|\hat d_{ijlp}(t)|\left(\mu_{ijlp}^{(1)}\xi_{ijp}\ov{a}_j\dst\frac{d_j}{d_i}+\mu_{ijlp}^{(2)}\til{\xi}_{ilp}\ov{a}_l\frac{d_l}{d_i}\right)\right)<0,\, i=1,\ldots,n,\, t\in[0,\omega],
   \end{align}

  then every solution of the system \eqref{model-studied in our work} with bounded initial condition, $x(t)$, satisfies 
   $$
  \lim\limits_{t \to +\infty}|x(t)-\hat{x}(t)|=0,
 $$
 where $\hat{x}(t)$ is the $\omega-$periodic solution of \eqref{ass-model-studied inour work}.
\end{corolario}
\begin{rem}
We note that the previous result establishes sufficient conditions for all solutions of \eqref{model-studied in our work}, which is not necessarily a periodic model, to converge to a periodic function as time goes to infinity. In the case of
 the Cohen--Grossberg neural network model \eqref{model-studied in our work} is not a periodic system, then the criterion in \cite[Theorem 5.2]{submitted} cannot be applied to prove the existence of a periodic solution of \eqref{model-studied in our work}. 
\end{rem}

As a particular example of model \eqref{model-studied in our work}, we have the following high-order Cohen--Grossberg neural network model,
\begin{eqnarray}\label{eq:H-O-Cohen-Grossberg-model}
	x_i'(t)\lefteqn{=a_i(x_i(t))\bigg[-b_i(x_i(t))+\sum_{j=1}^nc_{ij}(t)f_j(\rho_jx_j(t))+\sum_{j=1}^nd_{ij11}(t)f_{j}\left(\rho_j\int_0^{+\infty}K_{ij}(u)x_j(t-u)du\right)}\nonumber\\
	&+\left.\dst\sum_{j,l=1}^nd_{ijl2}(t)f_j\left(\rho_j\int_0^{+\infty}K_{ij}(u)x_j(t-u)du\right)f_l\left(\rho_l\int_0^{+\infty}K_{il}(u)x_l(t-u)du\right)+I_i(t)\right],\nonumber\\
\end{eqnarray}
for $\,t\geq0,\,i=1,\ldots,n,$ where, for each $i,j,l=1,\ldots,n$ and $q=1,2$, $\rho_i>0$, $a_i:\er\to(0,+\infty)$, $b_i:\er\to\er$, $c_{ij},d_{ijlq},I_i:[0,+\infty)\to\er$, $f_{j}:\er\to\er$, and $K_{ij}:[0,+\infty)\to[0,+\infty)$ are continuous functions such that $K_{ij}$ verifies \eqref{ex:kernel1}.
The existence and global exponential stability of a periodic solution of \eqref{eq:H-O-Cohen-Grossberg-model} were studied in \cite{liu2011periodic} (see also  \cite{submitted}).

Considering $\hat{c}_{ij},\hat{d}_{ijlq}:[0,+\infty)\to\er$ continuous functions such that
\begin{eqnarray}\label{ass cond for H-O}
		\lim\limits_{t\to+\infty} \big(c_{ij}(t)-\hat c_{ij}(t)\big)=\lim\limits_{t\to+\infty} \big(d_{ijlq}(t)-\hat d_{ijlq}(t)\big)=\lim\limits_{t\to+\infty} \big(I_i(t)-\hat I_{i}(t)\big)=0,
\end{eqnarray}
for each $i,j,l=1,\ldots,n$ and $q=1,2$, then \eqref{eq:H-O-Cohen-Grossberg-model} is an asymptotic system of

\begin{eqnarray}\label{eq:H-O-Cohen-Grossberg-model ass}
	x_i'(t)\lefteqn{=a_i(x_i(t))\bigg[- b_i(x_i(t))+\sum_{j=1}^n\hat c_{ij}(t)f_j(\rho_jx_j(t))+\sum_{j=1}^n\hat d_{ij11}(t)f_{j}\left(\rho_j\int_0^{+\infty}K_{ij}(u)x_j(t-u)du\right)}\nonumber\\
	&+\left.\dst\sum_{j,l=1}^n\hat d_{ijl2}(t)f_j\left(\rho_j\int_0^{+\infty}K_{ij}(u)x_j(t-u)du\right)f_l\left(\rho_l\int_0^{+\infty}K_{il}(u)x_l(t-u)du\right)+\hat I_i(t)\right].\nonumber\\
\end{eqnarray}
 Consequently, from Corollary \ref{cor 4.4}, we obtain the following result.

 \begin{corolario}\label{periodic soln exist for HO}
 Assume the following hypotheses. 
 
 For each $i=1,\dots,n$, and $q=1,2$:
 \begin{description}
 	\item[A1.] the functions $c_{ij}$, $d_{ijlq}$, and $I_{i}$ are $\omega-$periodic for some $\omega>0$;
     \item[A2.] there are $\un a_i, \ov a_i>0$ such that
            $$\un a_i<a_i(u)< \ov a_i;$$
     \item[A3.] there are $\beta_i, \beta_i^*>0$ such that                                            \begin{align*}
		\beta_i\leq\frac{b_i(u)-b_i(v)}{u-v}\leq \beta_i^*, \Es \forall u,v\in \er,\, u\neq v;
	\end{align*}
    \item[A4.] 	there exist $M_i,\mu_i>0$ such that
     $$
	  |f_i(u)-f_i(v)|\leq\mu_i|u-v|\Es\text{and}\Es |f_{i}(u)|\leq M_i,\Es\forall u,v\in\er.
	 $$
 \end{description}
 If there exists $d=(d_1,\ldots,d_n)>0$ such that, for all 	$t\in[0,\omega]$ and $i=1,\ldots,n$,	
	\begin{eqnarray}\label{4.3f}
		\beta_i>\sum_{j=1}^{n}\left[\left(|c_{ij}(t)|+|d_{ij11}(t)|\right)\mu_{j}\rho_j\frac{d_j}{d_i}
		+\sum_{l=1}^n|d_{ijl2}(t)|\left(\frac{d_j}{d_i}M_l\mu_{j}\rho_j+\frac{d_l}{d_i}M_j\mu_{l}\rho_l\right)\right],
	\end{eqnarray}
	then the model \eqref{eq:H-O-Cohen-Grossberg-model} has an $\omega-$periodic solution.
	
		If, in additional to  condition \eqref{4.3f}, there exists $d^*=(d^*_1,\ldots,d^*_n)>0$ such that, for all 	$t\in[0,\omega]$ and $i=1,\ldots,n$,
    \begin{eqnarray}\label{cor-cond-est-model-LX}
\un{a}_i\beta_i>\sum_{j=1}^{n}\bigg[\left(|c_{ij}(t)|+|d_{ij11}(t)|\right)\mu_{j}\rho_j\ov{a}_j\frac{d^*_j}{d^*_i}+\sum_{l=1}^n|d_{ijl2}(t)|\big(\frac{d^*_j}{d^*_i}\ov{a}_jM_l\mu_{j}\rho_j+\frac{d^*_l}{d^*_i}\ov{a}_lM_j\mu_{l}\rho_l\big)\bigg],
	\end{eqnarray}
 then every solution of the system \eqref{eq:H-O-Cohen-Grossberg-model ass} with bounded initial condition, $\hat{x}(t)$, satisfies
  $$
  \lim\limits_{t \to +\infty}|\hat{x}(t)-\tilde{x}(t)|=0,
 $$
 where $\tilde{x}$ is the $\omega-$periodic solution of \eqref{eq:H-O-Cohen-Grossberg-model}.
\end{corolario}

\begin{proof}
	In the model \eqref{model-studied in our work}, taking $P=1$, $Q=2$ and, for each $i,j,l=1,\ldots,n$, $q=1,2$, $u,u_1,u_2\in\er$, and $t\geq0$, let $a_i(t,u)=a_i(u)$, $b_i(t,u)=b_i(u)$,  $F_i(u)=H_i(u)=u$, $\tau_{ij1}(t)=\til{\tau}_{ij1}(t)=0$, $h_{ijl1}(u_1,u_2)=f_j(\rho_ju_1)$, $c_{ij11}(t)=c_{ij}(t)$, $c_{ijl1}(t)=d_{ijl1}(t)=0$ for $l\neq1$,  $f_{ijl1}(u_1,u_2)=f_j(\rho_ju_1)$, $f_{ijl2}(u_1,u_2)=f_j(\rho_ju_1)f_l(\rho_lu_2)$, $g_{ijq}(u)=\til{g}_{ijq}(u)=u$, and $\til{\eta}_{ijq}(s)=\eta_{ijq}(s)=\int_{-\infty}^0K_{ij}(-v)dv$, for $s\leq0$, we obtain model \eqref{eq:H-O-Cohen-Grossberg-model}.

	Trivially hypotheses H6. and $i.$ in Theorem \ref{periodic-soln-exist} hold. From A1., A2., and A3.,  hypotheses $ii.$ and $iii.$ in Theorem \ref{periodic-soln-exist} also hold.
	
	For all $u_1,u_2,v_1,v_2\in\er$, we have
	$$
	|h_{ijl1}(u_1,u_2)-h_{ijl1}(v_1,v_2)|=|f_j(\rho_ju_1)-f_j(\rho_jv_1)|\leq\rho_j\mu_j|u_1-v_1|,
	$$
	$$
	|f_{ijl1}(u_1,u_2)-f_{ijl1}(v_1,v_2)|=|f_j(\rho_ju_1)-f_j(\rho_jv_1)|\leq\rho_j\mu_j|u_1-v_1|,
	$$
	and
	\begin{align*}
		|f_{ijl2}(u_1,u_2)-f_{ijl2}(v_1,v_2)|&=|f_j(\rho_ju_1)f_l(\rho_lu_2)-f_j(\rho_jv_1)f_l(\rho_lv_2)|\\
		&\leq |f_j(\rho_ju_1)-f_j(\rho_jv_1)||f_l(\rho_lu_2)|+|f_j(\rho_jv_1)||f_l(\rho_lu_2)-f_l(\rho_lv_2)|
		\\
		&\leq M_l\rho_j\mu_j|u_1-v_1|+M_j\rho_l\mu_l|u_2-v_2|,
	\end{align*}
	for all $i,j,l=1,\ldots,n$, thus hypothesis H5. holds. Condition \eqref{4.3b} follows from \eqref{4.3f}, and the inequality \eqref{cor:cond-est-modelo-AJC} reads as \eqref{cor-cond-est-model-LX}. Finally, the result follows from Corollary \ref{cor 4.4}.
\end{proof}

\begin{rem} 
	We emphasize that since \eqref{eq:H-O-Cohen-Grossberg-model} is an asymptotic system of itself, its periodic solution, $\tilde{x}(t)$, serves as a global attractor for  all solutions with bounded initial conditions. The global exponential stability of  $\tilde{x}(t)$ is proved in \cite[Corrolary 5.9]{submitted}, given the additional condition $\int_0^{+\infty}K_{ij}(u)\Ne^{\alpha u}du<+\infty$, for some $\alpha>0$.
\end{rem}

\exemplo Consider the following low-order Cohen--Grossberg neural network model
\begin{eqnarray}\label{eq:L-O-Cohen-Grossberg-model}
	x_i'(t)=a_i(x_i(t))\bigg[-\lefteqn{ b_i(x_i(t))+\sum_{j=1}^n c_{ij}(t)f_j(x_j(t))+\sum_{j=1}^n d_{ij}(t)g_{j}\left(x_j(t-\tau_{ij}(t))\right)}\nonumber\\
	&+\left.\dst\sum_{j=1}^n p_{ij}(t)\int_0^{+\infty}K_{ij}(u)h_j(x_j(t-u))du+I_i(t)\right],
\end{eqnarray}
for $\,t\geq0,\,i=1,\ldots,n,$ where, for each $i,j=1,\ldots,n$, $a_i:\er\to(0,+\infty)$, $b_i:\er\to\er$, $c_{ij},d_{ij},p_{ij},I_i:[0,+\infty)\to\er$, $\tau_{ij},K_{ij}:[0,+\infty)\to[0,+\infty)$, and $f_{j},g_j,h_j:\er\to\er$ are continuous functions.% such that $K_{ij}$ verifies \eqref{ex:kernel1}.

The existence and global exponential stability of an almost periodic solution of \eqref{eq:L-O-Cohen-Grossberg-model} is studied in \cite{xiang+cao}.

For model \eqref{eq:L-O-Cohen-Grossberg-model}, the following hypotheses will be considered:
\begin{description}
	\item[B1.] There is $\mu>0$ such that for each $i,j=1,\ldots,n$,
	$$
	  \int_0^{+\infty}K_{ij}(t)dt=1\Es\text{and}\Es\int_0^{+\infty}\e^{\mu t}K_{ij}(t)dt<+\infty;
	$$
	\item[B2.] For each $i=1,\ldots,n$,  there are $\un a_i, \ov a_i>0$ such that
	$$\un a_i<a_i(u)< \ov a_i,\Es \forall u\in \er;$$
	\item[B3.] For each $i=1,\ldots,n$, there is $\beta_i>0$ such that                              
	\begin{align*}
		\beta_i\leq\frac{b_i(u)-b_i(v)}{u-v}, \Es u,v\in \er,\, u\neq v;
	\end{align*}
    \item[B4.] For each $i,j=1,\ldots,n$, the function $\tau_{ij}$ is almost periodic;
	\item[B5.] For each $i=1,\ldots,n$, $f_i(0)=g_i(0)=h_i(0)=0$ and there exist $L_i^f,L_i^g,L_i^h>0$ such that
	$$
	|f_i(u)-f_i(v)|\leq L_i^f|u-v|,\,\,|g_i(u)-g_i(v)|\leq L_i^g|u-v|,\,\,\text{and}\,\, |h_i(u)-h_i(v)|\leq L_i^h|u-v|,
	$$
	for all $u,v\in\er$.
\end{description}

\begin{teorema}\label{teo-almost-periodic-model}
	Assume B1.-B5..
	
	If there are $\tilde{c}_{ij},\tilde{d}_{ij},\tilde{p}_{ij}, \tilde{I}_i:\er\to\er$ continuous and almost periodic functions on $\er$ such that: 
	\begin{enumerate}
		\item $\dst\lim_{t\to+\infty}(c_{ij}(t)-\tilde{c}_{ij}(t))=\lim_{t\to+\infty}(d_{ij}(t)-\tilde{d}_{ij}(t))=\lim_{t\to+\infty}(p_{ij}(t)-\tilde{p}_{ij}(t))\\=\lim_{t\to+\infty}(I_{i}(t)-\tilde{I}_{i}(t))=0$;
		\item  there exists $d=(d_1,\ldots,d_n)>0$ verifying
	$$
	   d_i\un{a}_i\beta_i>\sum_{j=1}^n\ov{a}_jd_j\left(\tilde{c}_{ij}^+L_j^f+\tilde{d}_{ij}^+L_j^g+\tilde{p}_{ij}^+L_j^h\right),\Es i=1,\ldots,n,
	$$  
	where $\tilde{c}_{ij}^+=\dst\sup_{t\in\er}|\tilde{c}_{ij}(t)|$, $\tilde{d}_{ij}^+=\dst\sup_{t\in\er}|\tilde{d}_{ij}(t)|$, and $\tilde{p}_{ij}^+=\dst\sup_{t\in\er}|\tilde{p}_{ij}(t)|$,
\end{enumerate}
     then there exists an almost periodic function $\tilde{x}:\er\to\er$ such that
	$$
	  \lim_{t\to+\infty}|x(t)-\tilde{x}(t)|=0,
	$$
	for all solution $x(t)$ of model \eqref{eq:L-O-Cohen-Grossberg-model} with bounded initial condition.
\end{teorema}
\begin{proof}
Clearly, system \eqref{eq:L-O-Cohen-Grossberg-model} is a particular situation of model \eqref{1}.

By condition i. and Definition \ref{def-asymp-system}, the system
\begin{eqnarray}\label{eq:L-O-Cohen-Grossberg-model2}
	x_i'(t)=a_i(x_i(t))\bigg[-\lefteqn{ b_i(x_i(t))+\sum_{j=1}^n \tilde{c}_{ij}(t)f_j(x_j(t))+\sum_{j=1}^n \tilde{d}_{ij}(t)g_{j}\left(x_j(t-\tau_{ij}(t))\right)}\nonumber\\
	&+\left.\dst\sum_{j=1}^n \tilde{p}_{ij}(t)\int_0^{+\infty}K_{ij}(u)h_j(x_j(t-u))du+\tilde{I}_i(t)\right],
\end{eqnarray}
is an asymptotic system of \eqref{eq:L-O-Cohen-Grossberg-model}.

From \cite[Theorem 1]{xiang+cao}, system \eqref{eq:L-O-Cohen-Grossberg-model2} has a unique almost periodic solution. Denote it by $\tilde{x}(t)$.

From B1.-B5., hypotheses H1.-H6. hold. From condition ii., we conclude that
$$
\limsup_{t\to+\infty}\left[-\un{a}_i\beta_i+\sum_{j=1}^n\ov{a}_j\frac{d_j}{d_i}\left(c_{ij}(t)L_j^f+d_{ij}(t)L_j^g+p_{ij}(t)L_j^h\right)\right]<0,
$$
and H7. holds. Now the conclusion follows from Theorem \ref{convegrence of solutions theorem}.
\end{proof}

\begin{rem}
	We note that the function $\tilde{x}(t)$ may not be a solution of the model  \eqref{eq:L-O-Cohen-Grossberg-model}. Moreover, the coefficient functions $c_{ij},d_{ij},p_{ij}$ and the inputs $I_i$ are not necessarily almost periodic functions, thus \cite[Theorem 2]{xiang+cao} can not be used to prove the global attractivity of \eqref{eq:L-O-Cohen-Grossberg-model}.
\end{rem}

\section{Numerical Example}
Here, we present a numerical example to illustrate the effectiveness of new results presented in this work.

{{
		 \exemplo Consider the following numerical example which is a particular situation of \eqref{1}.
   \begin{align}\label{non-periodic ex}	 
			x_1'(t)=&\left(\sin\big(x_1(t)\big)+2\right)\bigg[-(4+\Ne^{-t})x_1(t)\Ne^{\frac{\sin(x_1(t))}{1+x_1(t)^2}}\nonumber\\
   &+\left(\frac{1}{3}\cos(t)+\Ne^{-t}\right)\tanh\big(x_2(t-|\sin(t)|)\big)+\Ne^{\sin(t)}+\Ne^{-t}\bigg]\nonumber\\
   \\
			x_2'(t)=&\left(\cos\big(x_2(t)\big)+2\right)\bigg[-(5+\cos(t)+\Ne^{-t})x_2(t)\nonumber\\
  & +\left(\frac{2}{3}\sin(t)+\Ne^{-t}\right)\tanh\big(x_1(t-|\cos(t)|)\big)+\cos(t)+\Ne^{-t}\bigg]\nonumber
	\end{align}
Here $n=2$, $P=1$, $a_1(t,u)=\sin(u)+2$, $a_2(t,u)=\cos(u)+2$, $b_2(t,u)=(5+\cos(t)+\Ne^{-t})u$, $b_1(t,u)=(4+\Ne^{-t})u\Ne^{\frac{\sin(u)}{1+u^2}}$, $F_1(u)=F_2(u)=u$, $c_{1211}(t)=\frac{1}{3}\cos(t)+\Ne^{-t}$, $c_{2111}(t)=\frac{2}{3}\sin(t)+\Ne^{-t}$, $h_{1211}(u,v)=h_{2111}(u,v)=\tanh(u)$, $\tau_{121}(t)=|\sin(t)|$, $\tau_{211}(t)=|\cos(t)|$, $I_1(t)=\Ne^{\sin(t)}+\Ne^{-t}$, $I_2(t)=\cos(t)+\Ne^{-t}$, and $c_{11j1}(t)=c_{22j1}(t)=d_{ijl1}(t)=0$.}}

\begin{figure}[ht]		
	\centering
	{\includegraphics[width=15cm]{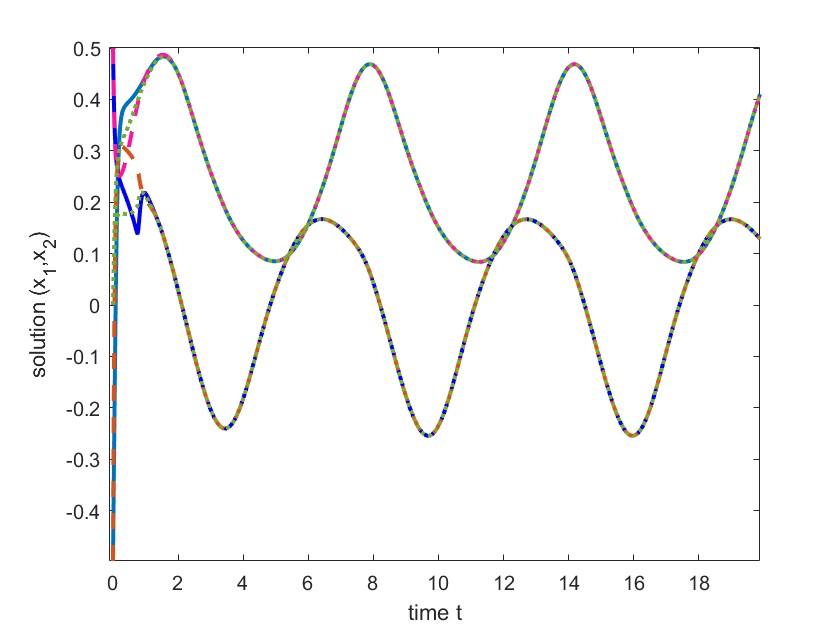}}
	\caption{{{Numerical simulation of three solutions $(x_1(t),x_2(t))$ of system \eqref{non-periodic ex}, with initial condition $\varphi(s)=(-\Ne^s/2,\cos(s)/2)$, $\varphi(s)=(\cos(s)/2,-\Ne^s/2)$, $\varphi(s)=(\sin(s),\Ne^s-1)$ for $s\leq0$, respectively, at $t_0=0$.}}}\label{figura1}
\end{figure}
\begin{figure}[ht]	
	
	\centering
	{\includegraphics[width=14cm]{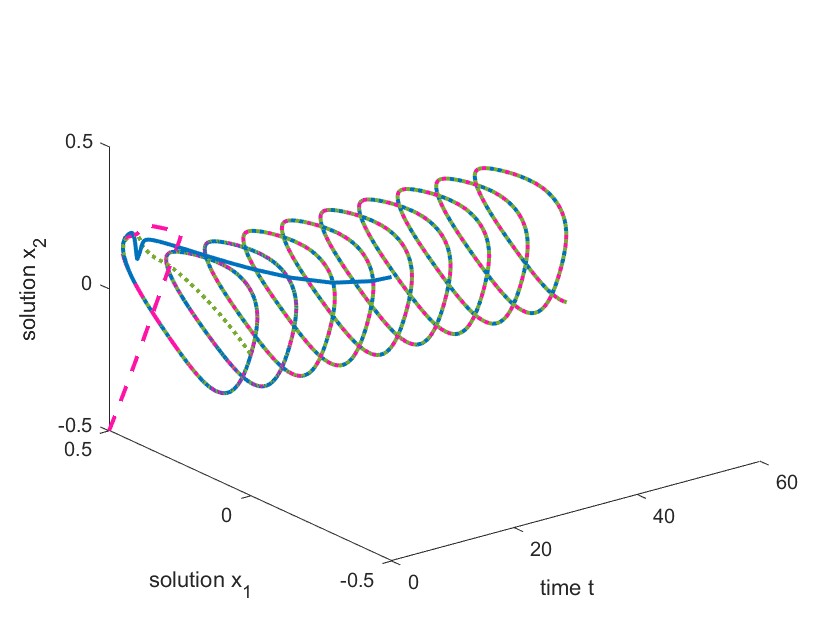}}
	\caption{{{Numerical simulation of three solutions $(x_1(t),x_2(t))$ of system \eqref{non-periodic ex}, with initial condition $\varphi(s)=(-\Ne^s/2,\cos(s)/2)$, $\varphi(s)=(\cos(s)/2,-\Ne^s/2)$, $\varphi(s)=(\sin(s),\Ne^s-1)$ for $s\leq0$, respectively, at $t_0=0$.}}}\label{figura2}
\end{figure}

System \eqref{non-periodic ex} has the following system 
\begin{align}\label{non-periodic ex2}	 
	\hat{x}_1'(t)=&\left(\sin\big(\hat{x}_1(t)\big)+2\right)\bigg[-4x_1(t)\Ne^{\frac{\sin(\hat{x}_1(t))}{1+\hat{x}_1(t)^2}}+\frac{1}{3}\cos(t)\tanh\big(\hat{x}_2(t-|\sin(t)|)\big)+\Ne^{\sin(t)}\bigg]\nonumber\\
	\\
	\hat{x}_2'(t)=&\left(\cos\big(\hat{x}_2(t)\big)+2\right)\bigg[-(5+\cos(t))\hat{x}_2(t) +\frac{2}{3}\sin(t)\tanh\big(\hat{x}_1(t-|\cos(t)|)\big)+\cos(t)\bigg]\nonumber
\end{align}
as one of its asymptotic systems. As hypotheses H1.-H7. hold, Theorem \ref{convegrence of solutions theorem} allowed us to conclude that
$$
  \lim_{t\to+\infty}|x(t)-\hat{x}(t)|=0,
$$
for all  solutions $x(t)=(x_1(t),x_2(t))$ and $\hat{x}(t)=(\hat{x}_1(t),\hat{x}_2(t))$, of \eqref{non-periodic ex} and \eqref{non-periodic ex2} respectively,	with bounded initial conditions. Noting that \eqref{non-periodic ex2} is a $2\pi-$periodic system, \cite[Theorem 5.2]{submitted} assures that \eqref{non-periodic ex2} has a $2\pi-$periodic solution $\tilde{x}(t)=(\tilde{x}_1(t),\tilde{x}_2(t))$. Thus 
$$
  \lim_{t\to+\infty}|x(t)-\tilde{x}(t)|=0,
$$
for all solutions $x(t)$ of \eqref{non-periodic ex} with bounded initial conditions. We stress that all solutions of \eqref{non-periodic ex} converge asymptotically to $\tilde{x}(t)$ which is not a solution of \eqref{non-periodic ex}.

Figures \ref{figura1} and \ref{figura2} show the plot of three solutions $(x_1(t),x_2(t))$ of system \eqref{non-periodic ex}, with initial condition $\varphi(s)=(-\Ne^s/2,\cos(s)/2)$, $\varphi(s)=(\cos(s)/2,-\Ne^s/2)$, $\varphi(s)=(\sin(s),\Ne^s-1)$ for $s\leq0$, respectively, at $t_0=0$.
\section{Conclusion}
In this paper, we have presented sufficient conditions for the global convergence of asymptotic systems in high-order Cohen--Grossberg neural network (CGNN) models with infinite discrete time-varying and distributed delays, as shown in Theorem \ref{convegrence of solutions theorem}. This represents a significant extension of the previous results found in \cite{oliveira2017convergence,xiao+zhang,yuan+huang+hu+liu,yuan+yuan+huang+hu,zhou+li+zhang}.

Model \eqref{1} is general enough to include low-order CGNN models, and a global attractive criterion is obtained for the CGNN model \eqref{eq:static-model with k}, as stated in Corollary \ref{cor-crit-static-model}. This new result enhances the previous criterion found in \cite{ncube2020existence} for the static neural network model \eqref{eq:static-model-INcube}.  In Example \ref{high-order-ex}.2, we derived sufficient conditions for all solutions of the high-order CGNN model \eqref{model-studied in our work}, with bounded initial conditions, to converge to a periodic function as time goes to infinity, as detailed in Corollary \ref{cor 4.4}. It is relevant to note that 
 \cite[Theorem 5.2]{submitted} cannot be applied to prove the existence of a periodic solution of \eqref{model-studied in our work} because \eqref{model-studied in our work} is not necessarily a periodic model.
 
 Finally, in Example  4.3, we obtain sufficient conditions for all solutions of CGNN model \eqref{eq:L-O-Cohen-Grossberg-model}, presented in \cite{xiang+cao}, with bounded initial conditions to converge to an almost periodic function as times goes to infinity, as stated in Theorem \ref{teo-almost-periodic-model}. It is worth noting that the coefficients in \eqref{eq:L-O-Cohen-Grossberg-model} are not necessarily almost periodic.

\noindent
{\bf Acknowledgments.}\\
This work was partially supported by Funda\c{c}\~ao para a Ci\^encia e a Tecnologia (Portugal) within the Projects UIDB/00013/2020, UIDP/00013/2020 of CMAT-UM (Jos\'e J. Oliveira), and Project UIDB/00212/2020 of CMA-UBI (A. Elmwafy and C\'esar M. Silva).

\end{document}